\documentclass[article,10pt]{amsart}
\usepackage{amsfonts}
\usepackage{amsmath}
\usepackage{amsthm}
\usepackage{amssymb}
\usepackage[dvipsnames,usenames]{color} %color
\usepackage{enumitem} % resume enumerating
\usepackage[cmtip,arrow]{xy}
\usepackage{pb-diagram,pb-xy}
\usepackage{enumitem}
\usepackage{xcolor}
\usepackage{mathrsfs}
\usepackage{verbatim} %Allows multiline comments
\usepackage{amscd}
\usepackage{anysize}
\usepackage{subfig}
\usepackage{array}
\usepackage{graphicx}
\usepackage[colorlinks=true, urlcolor=NavyBlue, linkcolor=NavyBlue, citecolor=NavyBlue]{hyperref} %Makes references/citations as links in pdf.
\usepackage{ulem} %Allows strikethrough

\newtheorem{theorem}{Theorem}[section]
\newtheorem{lemma}[theorem]{Lemma}
\newtheorem{proposition}[theorem]{Proposition}
\newtheorem{corollary}[theorem]{Corollary}

\newtheorem{definition}[theorem]{Definition}
\newtheorem{example}[theorem]{Example}
\newcommand{\ov}[1]{\overline{#1}}
\newcommand{\qalex}[1]{\mathcal{A}(#1)}
\newcommand{\sigdef}[ 2]{\sigma^{(2)}(#1, #2) - \sigma (#1)}
\newcommand{\bl}[0]{\mathcal{B}\ell}

\newcommand{\csum}[0]{{K_1\# -K_2}}
\newcommand{\alex}[1]{\mathcal{A}^{\mathbb{Z}}(#1)}
\newcommand{\rib}[0]{\mathfrak{R}}

\newcommand{\rcsum}[0]{{\mathfrak{R}\#-\mathfrak{R}}}
\newcommand{\ea}[0]{\eta_1}
\newcommand{\eb}[0]{\eta_2}

\newcommand{\mf}[1]{\mathfrak{#1}}
\newcommand{\defined}{\equiv}
\newcommand{\bdy}{\partial}

\newcommand{\p}{\mathcal{P}}
\newcommand{\Z}{\mathbb{Z}}
\newcommand{\Q}{\mathbb{Q}}

\begin{document}

\title{The Effect of Infecting Curves on Knot Concordance}

\author{Bridget D. Franklin$^{\dag}$}
\address{Department of Mathematics\\Rice University\\P.O. Box 1892\\Houston, TX 77251}
\email{bridget.franklin@rice.edu}
\urladdr{http://math.rice.edu/~bf2}
\date{\today}

\thanks{ $^{\dag}$Partially supported by Nettie S. Autry Fellowship}


%%%%%%%%%%%%%%%%%%%%%%%%%%%%%
%											%
%			Abstract and Introduction				%
%											%
%%%%%%%%%%%%%%%%%%%%%%%%%%%%%

\begin{abstract}
Various obstructions to knot concordance have been found using Casson-Gordon invariants, higher-order Alexander polynomials, as well as von-Neumann $\rho$-invariants.  Examples have been produced using (iterated) doubling operations $K\equiv \rib(\eta,J)$, and considering these as parametrized by invariants of the base knot $J$ and doubling operator $\rib$.  In this paper, we introduce a new mew method to obstruct concordance.  We show that infinitely many distinct concordance classes may be constructed by varying the infecting curve $\eta$ in $S^3-\rib$. Distinct concordance classes are found even while fixing the base knot, the doubling operator, and the order of $\eta$ in the Alexander module.
\end{abstract}

\maketitle

\section{Introduction}\label{intro}
A \emph{knot} is a (smooth) embedding of $S^1$ into $S^3$.  Two knots, $K_0 \subset S^3 \times \{0\}$ and $K_1 \subset S^3 \times \{1\}$ are said to be \emph{concordant} if there exists a (smooth) embedding of an annulus into $S^3 \times [0,1]$, $h: S^1 \times [0,1] \hookrightarrow S^3 \times [0,1]$, such that $h(S_1 \times \{i\}) = K_{i} \subset S^3 \times \{i\}$ for $i=0,1$.  The set of knots modulo concordance is known to form an abelian group under connected sum with identity element the trivial knot.  This group is the \emph{(smooth) concordance group of knots}, denoted $\mathcal{C}$.  If any knot is concordant to the trivial knot, this means it must bound a (smoothly) embedded disk in $S^3 \times [0,1]$, or, equivalently, in $B^4$, the four-dimensional ball bounded by $S^3$.  Knots of this form are \emph{slice knots}.  The complete structure of $\mathcal{C}$ is still not well understood.  In 1969, Levine defined an epimorphism from $\mathcal{C}$ onto a group of cobordism classes of Seifert matrices isomorphic to $ \mathbb{Z}^\infty \oplus \mathbb{Z}_2^{\infty}\oplus \mathbb{Z}_4^\infty$, called the \emph{algebraic concordance group} of knots \cite{Le10}.  Elements of the kernel of this map are \emph{algebraically slice knots}.

In order to better understand the structure of the (smooth) concordance group, Cochran, Orr, and Teichner defined a filtration of $\mathcal{C}$ by subgroups indexed by half integers
\[
	\dots \mf{F}_{n+1} \subset \mf{F}_{n.5} \subset \mf{F}_{n}\subset\dots\subset 
	\mf{F}_{0.5}\subset\mf{F}_{0}\subset\mathcal{C}.
\]
This is the \emph{$(n)$-solvable filtration}, or \emph{Cochran-Orr-Teichner (COT) filtration}, of the knot concordance group \cite{COT}.  Expanding on previous research by Levine \cite{Le10}, Milnor, and Casson-Gordon \cite{CG1, CG2}, Cochran-Harvey-Leidy show in \cite{CHL5, CHL3, CHL6} that for each $n\in \mathbb{Z}$
\[
	\mathbb{Z}^\infty \oplus \mathbb{Z}_2^\infty \subset \mf{F}_n/\mf{F}_{n.5}.
\]
Cochran-Harvey-Leidy create these families of linearly independent knots using iterated \textit{infections}, $K_{i+1} = \mf{R}_i(\eta_i,K_{i})$.  Here, each $\rib_i$ is a ribbon knot, and $\eta_i$ is some embedded oriented circle in $S^3- \rib_i$ which is unknotted in $S^3$ and has zero linking with $\rib$ (call such circles \textit{infecting curves}).  An example is shown n the left hand side of Figure \ref{fig:infection}.  We obtain $K_{i+1}$ by cutting the strands of $\rib_i$ which intersect the disk bounded by $\eta_i$ and ``tying them into the knot $K_i$" as in Figure \ref{fig:infection}.  This procedure will be explicitly defined in Section \ref{sec:background}.
\begin{figure}[h]
\includegraphics[scale=.5]{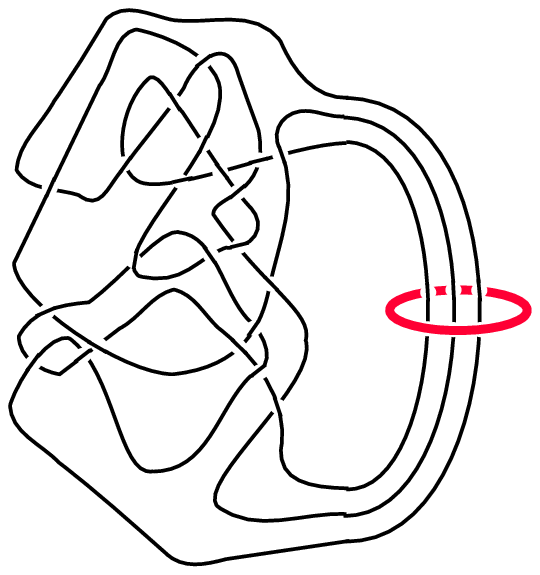}
\put(-100,50){$\rib_i$}\put(0,40){$\eta_i$}
\hspace{1in}
\includegraphics[scale=.5]{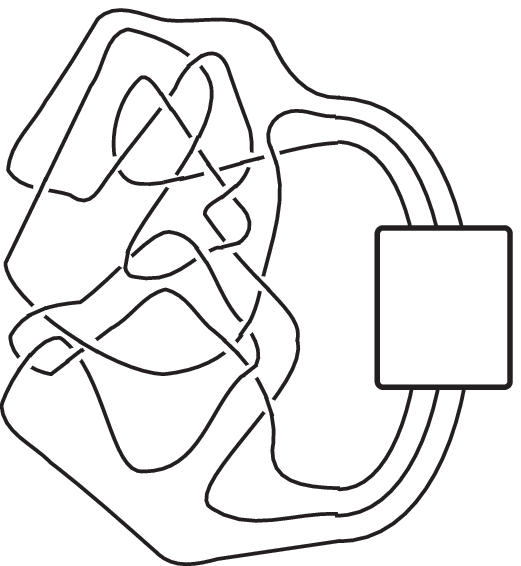}
\put(-100,50){$K_{i+1}$}\put(-17,35){$K_i$}
\caption{Infection: $K_{i+1}\equiv\rib_i(\eta_i,K_i)$}\label{fig:infection}
\end{figure}
Beginning with $K_0 \in \mf{F}_0$, this produces $K_i \in \mf{F}_i$ for each $i$ \cite[Proposition 2.7]{CHL5}.  At the $n$-th stage of the iteration, by varying the Levine-Tristram signatures of the inital knot $K_0$, Cochran-Harvey-Leidy produce an infinite rank subgroup $\mathbb{Z}^\infty \subset \mf{F}_n/\mf{F}_{n.5}$.  Next, they vary the ribbon knots $\mf{R}_i$, and more specifically their Alexander polynomials $p_i(t)$, to produce $n$ additional parameters for linear independence in each $\mf{F}_n/\mf{F}_{n.5}$.

Thus, many results regarding the structure of the $n$-solvable filtration have relied upon the classical signatures of the base knot $K_0$ and the Alexander polynomials of the ribbon knots $\rib_i$.  Here, we take a different approach by varying the infecting curves  while fixing the base knot $K_0$ and ribbon knots $\mf{R}_i$.  In particular, fixing $\rib$, we give conditions such that infection upon two distinct infecting curves, $\ea$ and $\eb \subset S^3-\rib$, by a knot $J$ yields distinct concordance classes.  If $\ea$ and $\eb$ have different orders in the Alexander module $\alex{\rib}$, the situation is easier -- a situation we treat in Section \ref{sec:1solv}.  Our main results, however, apply even when $\ea$ and  $\eb$ generate the same submodule of the Alexander module.   Furthermore, all the knots we produce are algebraically slice with vanishing Casson-Gordon invariants.  Obstructions are ultimately found using the \textit{Blanchfield linking form}, a sesquilinear form on the Alexander module of $\rib$.
\[
	\bl_\rib:\mathcal{A}^{\mathbb{Z}}(\rib)\times\mathcal{A}^{\mathbb{Z}}(\rib) \rightarrow {\mathbb{Q}(t)} \text{ mod }{\mathbb{Z}[t,t^{-1}]}.
\]
In particular, we consider the ``Blanchfield self-linking" of the infecting curve, $\bl(\eta,\eta)$.  Remark that we will often blur the distinction between infecting curves $\eta\subset S^3-\rib$ and the corresponding elements $[\eta]\in \mathcal{A}^\Z(\rib)$ (and ultimately any localized Alexander module), allowing $\eta$ to represent both.  Corollary \ref{cor:intense}, stated for simplicity, follows immediately from our main theorem (Theorem \ref{thm:distinct}).

%:	Simple Lemma
\newtheorem*{corollary:intense}{Corollary \ref{cor:intense}}
\begin{corollary:intense}
Suppose $\rib$ is any knot with $\Delta_\rib\neq 1$.  Then there exists a (countably infinite) set of infecting curves $\{\eta_i\}$ and also a knot $J$ such that each $K_i\equiv \rib(\eta_i,J)$ represents a distinct concordance classes in $\mathcal{C}$.
\end{corollary:intense}

Figure \ref{fig:946plot} plots the value of the Blanchfield self-linking of elements in the rational Alexander module for the specific ribbon knot $\rib=9_{46}$.  Note that the Alexander module of $\rib$ is $\Delta_\rib(t) = 2t^2 -5t+2$ and its rational Alexander module is cyclic, generated by $\eta$. Curves on the graph are given by $xy = c$ for $c \in \{\pm 1,\pm 2,\pm 3, \pm 4, \pm 5\} \subset \mathbb{Z}[1/2]$.  Each shaded point on the curve $xy=c$ is represented by an element of $\mathcal{A}^{\Z}(\rib)$ whose Blanchfield self-linking, as an element of the \textit{rational Alexander module}, is 
\[
	c\cdot \bl(\eta,\eta) \in \frac{\mathbb{Q}(t)}{\mathbb{Q}[t,t^{-1}]}.
\]
Furthermore, each of these elements in $\alex{\rib}$ is realized by an infecting curve, $\eta_c$, in $S^3-\rib$.  By the proof of Corollary \ref{cor:intense}, there exists a knot $J$ such that the infection $\rib(\eta_c,J)$ yields a distinct concordance class for each for each distinct $c\in \Z[1/2]$.  The verification of this graph will be given in Section \ref{sec:application1}

\begin{figure}[h]
\includegraphics[scale=.5]{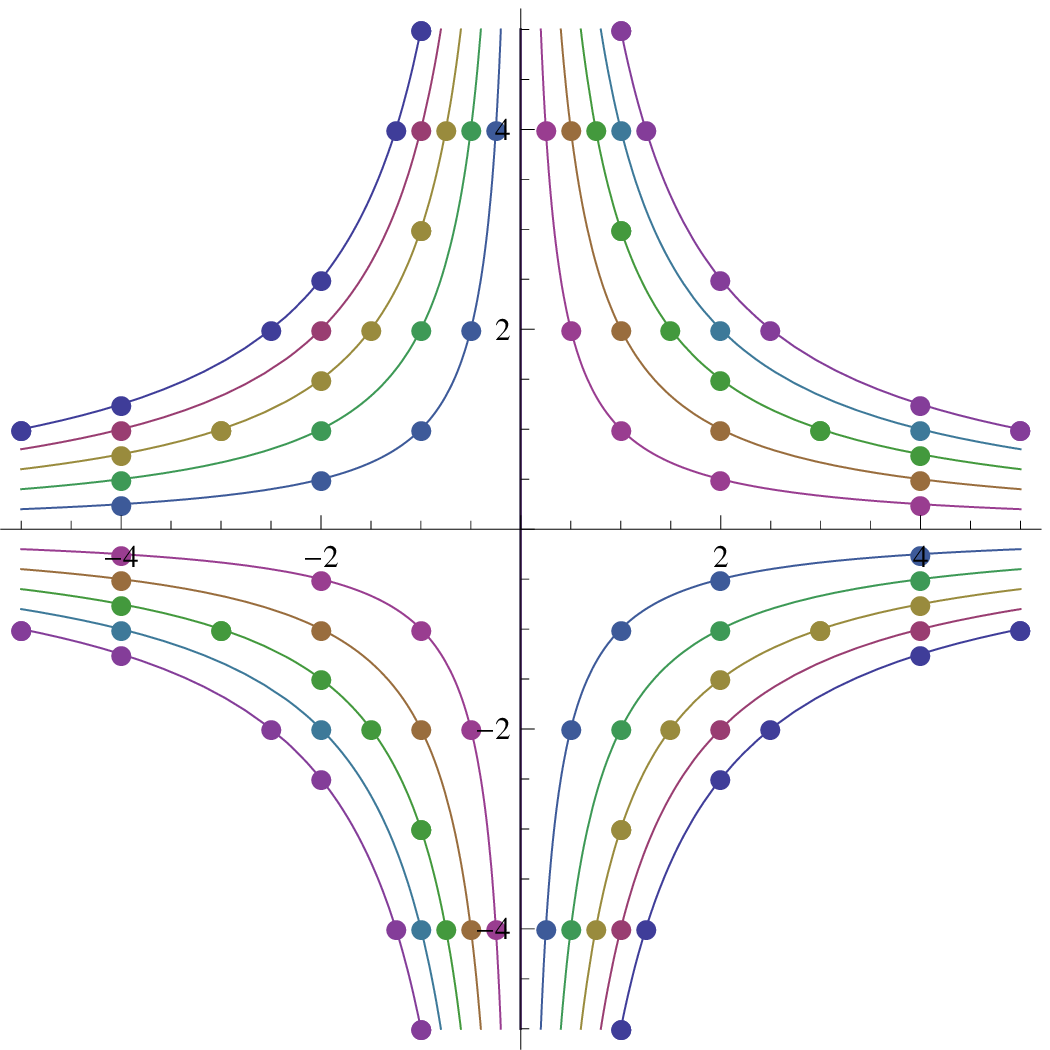}
\caption{}\label{fig:946plot}
\end{figure}

\newtheorem*{proposition:dense}{Proposition \ref{prop:dense}}
\begin{proposition:dense}
Let $\rib$ be a ribbon knot with Alexander polynomial $\Delta_\rib\neq 1$.  There exists a knot $J$ such that for any fixed infecting curve $\eta \subset S^3-\rib$,
\[
	\{ [\gamma] | \rib(\gamma, J) \text{ is concordant to } \rib(\eta,J) \text{ for some } \gamma \in [\gamma] \}
\]
is the subset of a quadric hypersurface in $\mathbb{Q}^{2g}$.
\end{proposition:dense}

Note that some restrictions on the infecting knot $J$ are necessary in the hypotheses of the reults presented.  In particular, if $J$ is slice, $\rib(\eta,J)$ will always be concordant to $\rib$.

Before beginning the technical details, we motivate our study with a few examples.  Certainly, it seems that infection upon distinct infecting curves should describe distinct satellite operations and therefore produce nonconcordant knots.  This is not always the case, however.

\begin{figure}[h]
	\includegraphics[scale=.2]{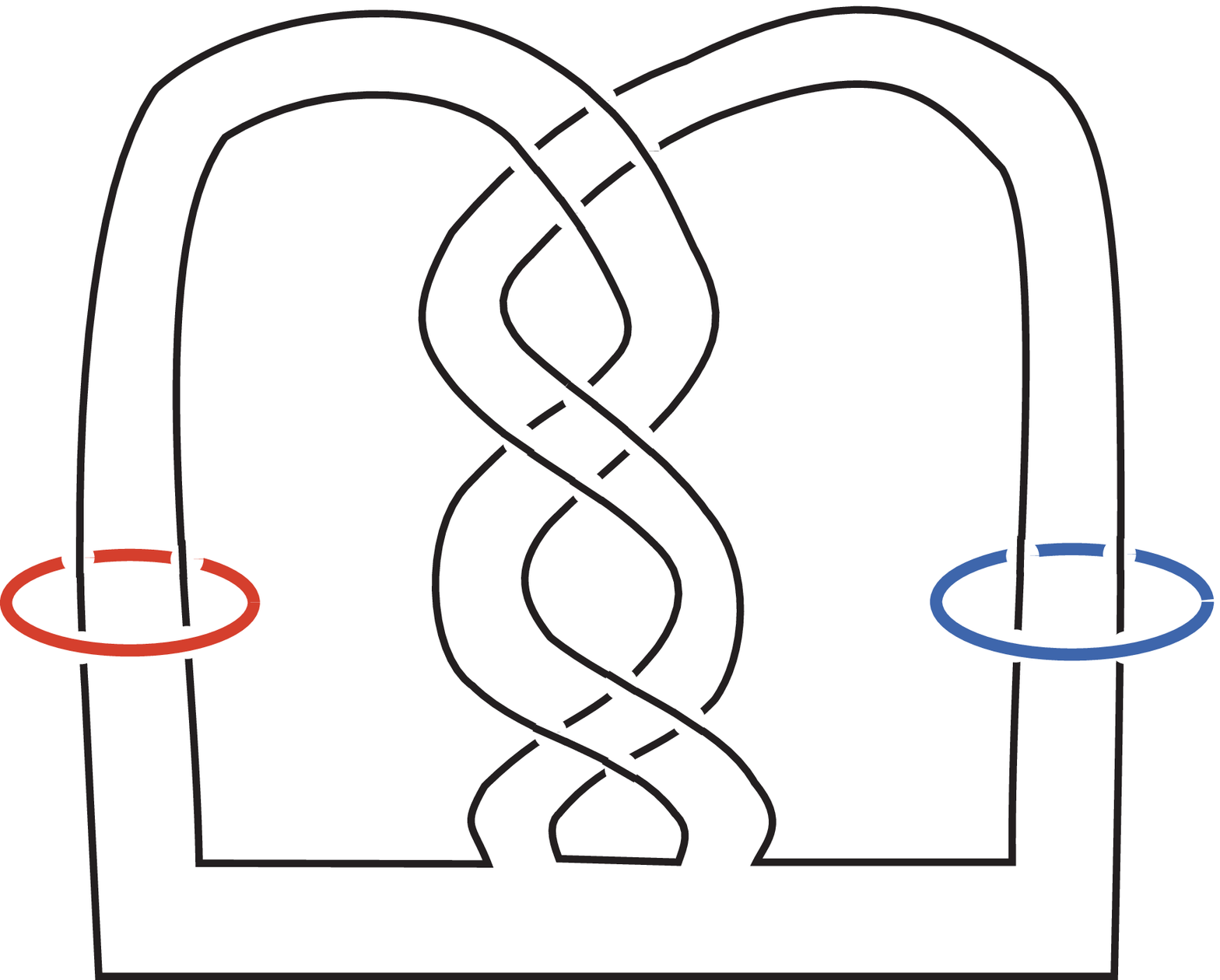}
	\put(-110,30){\textcolor{red}{$\alpha$}}
	\put(2,30){\color{blue}$\beta$}
	\caption{}\label{fig:946ab}%}
\end{figure}

\begin{example}\label{eg:ab}
Take $\rib_1=9_{46}$ and let $\alpha, \beta$ be the infecting curves shown in Figure \ref{fig:946ab}.  Take $J$ to be any knot, and set $K_1 \equiv \rib_1(\alpha,J)$,  $K_2\equiv \rib_1(\beta,J)$.  Notice that $\alpha$ and $\beta$, as elements in $\alex{\rib_1}$, have different orders and generate different submodules.  However, both $\alpha$ and $\beta$ encircle ribbon bands of $\rib$.  If we cut along the band encircled by $\alpha$ in $K_1$ we obtain a two-component trivial link, shown in Figure \ref{fig:946disks}, proving that $K_1$ is ribbon.  Similary, $K_2$ is also ribbon and $K_1$ and $K_2$ are concordant in $\mathcal{C}$.
\end{example}

\begin{figure}[h]
\parbox{1.5in}{
	\includegraphics[scale=.2]{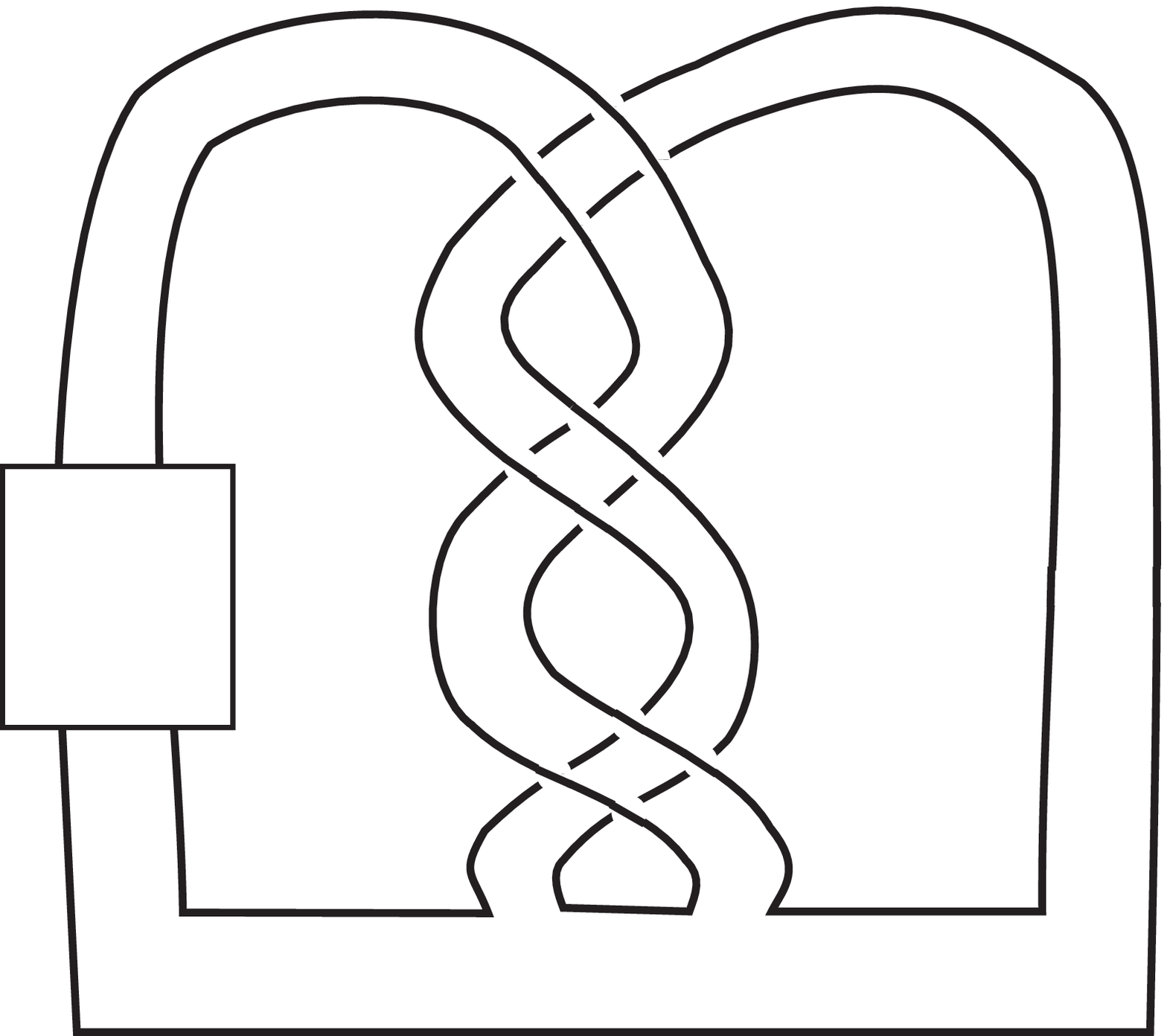}
	\put(-87,30){$J$}}
\parbox{1.5in}{
	\includegraphics[scale=.2]{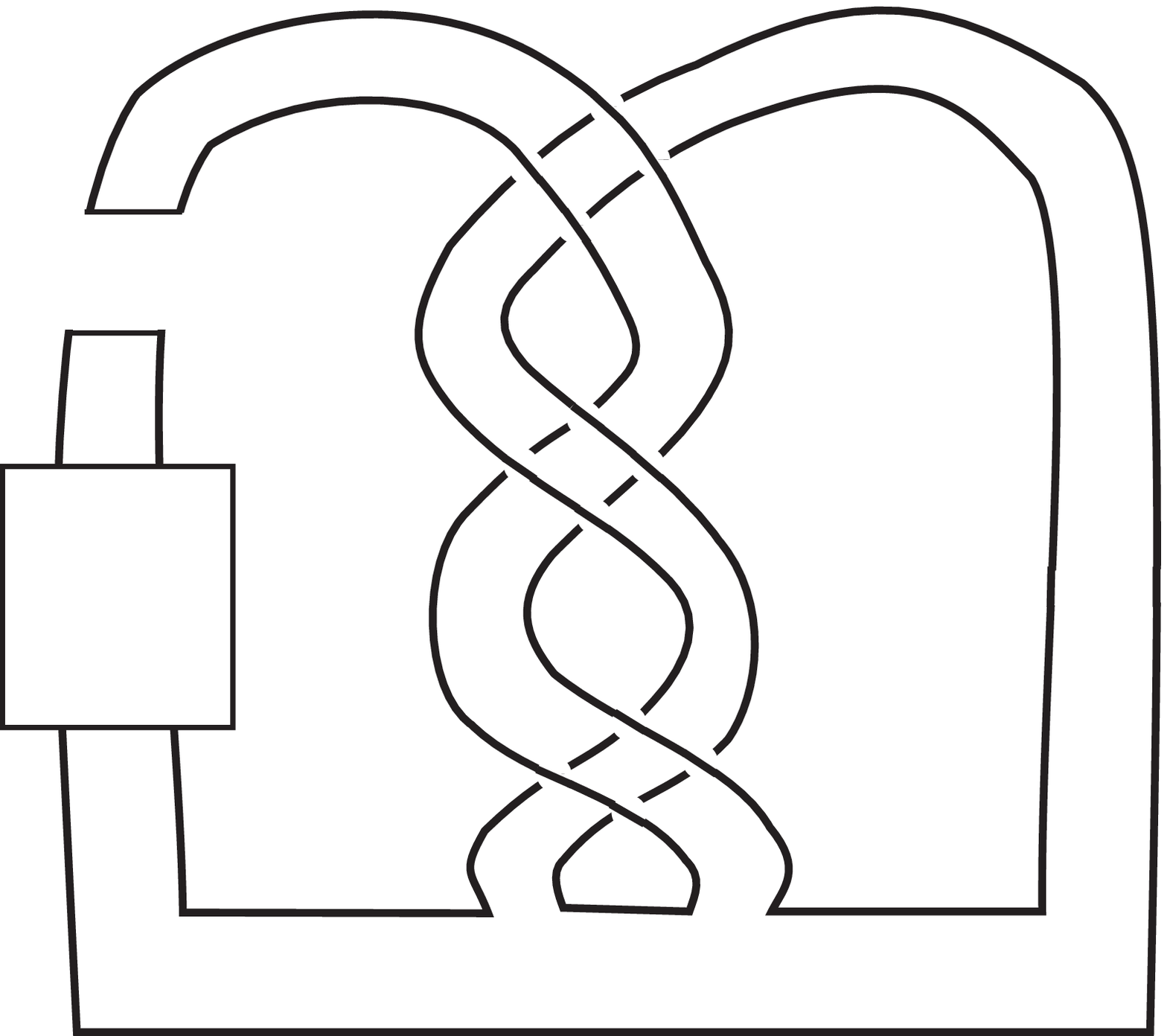}
	\put(-87,30){$J$}}
\parbox{1.5in}{
	\includegraphics[scale=.2]{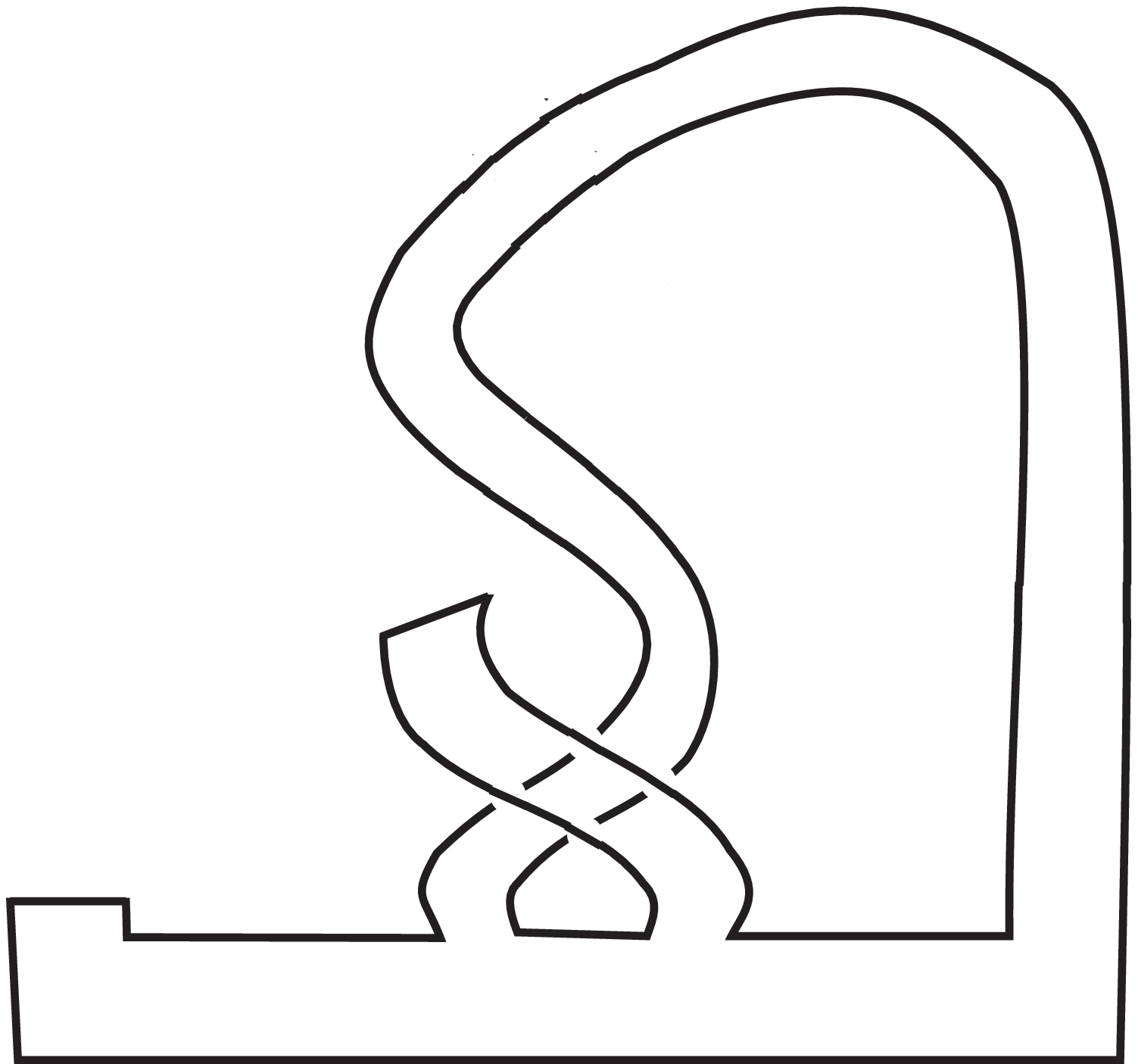}}
\caption{Cutting a ribbon band}\label{fig:946disks}
\end{figure}

In Example \ref{eg:ab}, $\alpha$ and $\beta$ generate different submodules and have different orders as elements of $\alex{\rib_1}$. However, both submodules are \textit{isotropic} with respect to the Blanchfield form, that is $\bl(\alpha,\alpha) = \bl(\beta,\beta)=0$.  This motivates our inquiry into how restrictions on the Blanchfield form could obstruct concordance between knots obtained by infecting $\rib$ using distinct infecting curves.  These restrictions prove lucrative even if the infecting curves generate the same submodule and have the same order in the Alexander module, $\mathcal{A}(\rib)$.  The following example illustrates that the question of which $\eta$ lead to distinct concordance classes is complicated even when $\bl(\eta,\eta)\neq 0$.  

\begin{figure}[h]
\subfloat[][]{\label{fig:sumrk1}
	\includegraphics[scale=1.3]{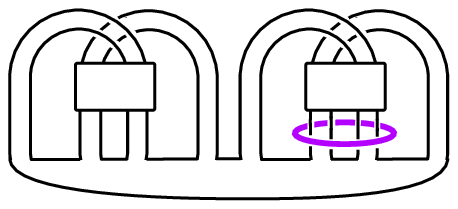}
	\put(-44,48){$2$}\put(-130,48){$1$}\put(-70,35){$\ea$}}
\subfloat[][]{\label{fig:sumrk2}
	\includegraphics[scale=1.3]{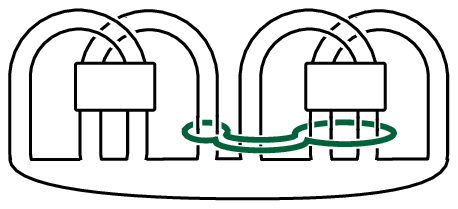}
	\put(-44,48){$2$}\put(-130,48){$1$}\put(-70,40){$\eb$}}
\caption{$\rib_2 = R_1\# R_2$}\label{fig:sumrk}
\end{figure}

\begin{example}\label{ex2}
Let $\rib_2 = R_1 \#R_2$ be the ribbon knot of Figure \ref{fig:sumrk} formed by taking the connected sum of ribbon knots $R_1, R_2$.  Here, the $1$ and $2$ inside the boxes indicate $1$ and $2$ negative full twists.  Let $\ea$ and $\eb$ be the infecting curves shown in Figures \ref{fig:sumrk1}, \ref{fig:sumrk2} respectively.  Again, take $J$ to be any knot, and set $K_1 \equiv \rib_2 (\ea, J)$, $K_2 \equiv \rib_2 (\eb,J)$. The Alexander module of $\rib_2$ is given by
\begin{equation*}
	\alex{\rib_2} = \alex{R_1}\oplus\alex{R_2} =\frac{\mathbb{Q}[t,t^{-1}]}{(1 - 2t)(2-1t)}\oplus \frac{\mathbb{Q}[t,t^{-1}]}{(2-3t)(3-2t)}
\end{equation*}
Notice that $\ea$ and $\eb$ generate different submodules of $\alex{\rib_2}$, neither of which is isotropic.  The orders of $[\ea], [\eb]$ are $(2-3t)(3-2t)$ and $(2-3t)(3-2t)(2-t)$ respectively.  However, the Blanchfield form yields
\begin{equation*}
	\bl(\ea,\ea) = \frac{5 (-1 + t)^2}{6 - 13 t + 6 t^2} = \bl(\eb,\eb).
\end{equation*}
In fact, $K_1$ and $K_2$ are concordant! This is because the ``extra band" encircled by $\eb$ is a ribbon band of $R_1$.  By cutting this band, we see that $\rib_2$ is concordant to $R_2$, shown in Figure \ref{fig:rk2}.  Thus, in a similar process to that depicted in Figure \ref{fig:946disks}, both $K_1$ and $K_2$ are concordant to $R_2(\eta,J)$.

\begin{figure}[h]
	\includegraphics[scale=.7]{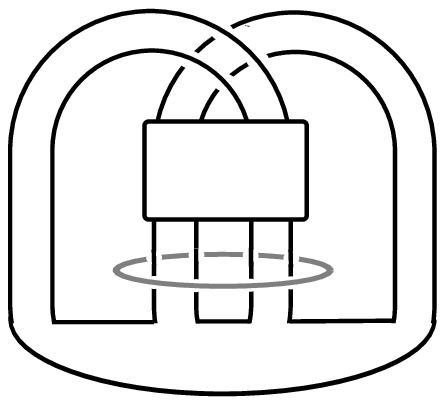}
	\put(-47,53){$2$}\put(-73,40){$\eta$}
	\caption{$R_2$}\label{fig:rk2}
\end{figure}
\end{example}

In Section \ref{sec:background}, we provide the necessary terminology and preliminaries before formally introducing our main theorem in Section \ref{sec:main}.  Also in Section \ref{sec:main}, we will also give a brief outline of its proof leaving some gaps.  In \ref{sec:bforms} we go through the technical details to fill in these gaps, giving obstructions created by the Blanchfield form.  Then, in Section \ref{sec:application1}, we give a complete study of distinct concordance classes, $K_i\equiv \rib(\eta_i,J)$, which may be produced by fixing the knot $J$ and setting $\rib \equiv 9_{46}$ as denoted in the Rolfsen knot tables \cite{R}.  This gives the necessary validation of Figure \ref{fig:946plot} as well as a detailed application of Corollary \ref{cor:intense}.  Finally, in Section \ref{sec:1solv}, we discuss an easier situation in which infecting curves have different orders in the Alexander module.  In particular, this allows us to slightly ease the hypotheses on the infecting knot $J$.

\section{Background}\label{sec:background}
\subsection{The $n$-solvable filtration of Cochran-Orr-Teichner}
In order to establish our results, culminating in Corollary ref{cor:intense} and Proposition \ref{prop:dense}, we rely on methods based on the \textit{$n$-solvable filtration} of the knot concordance group.  Our methods ultimately prove a stronger result than knots $K_1\equiv \rib(\ea,J)$ and $K_2\equiv \rib(\eb,J)$ being nonconcordant, rather we show that $\csum$ is $2$- but not $2.5$-solvable according to this filtration.  For any knot $K$, we denote by $M_K$ the closed $3$-manifold obtained by zero-surgery on $K$ in $S^3$.

%: 	n-solvable defined
\begin{definition}\cite[Definition 2.3]{CHL5}\label{def:nsolvable}
A knot $K$ is \textit{$n$-solvable} if there exists a spin $4$-manifold $V$ with boundary $\partial V = M_K$ such that
\begin{enumerate}
\item \label{prop:inclusion} Inclusion induces an isomorphism $H_1(M_K;\mathbb{Z}) \overset{\cong}\longrightarrow H_1(V;\mathbb{Z})$.
\item \label{prop:basis}There is a basis for $H_2(V;\mathbb{Z})$, $\{L_i,D_j|i,j=1,\dots,r\}$, consisting of compact, connected, embedded surfaces with trivial normal bundles which are pairwise disjoint, except that for each $i$, $L_i$ intersects $D_i$ transversely once with positive sign.
\item \label{prop:nthterm} Inclusion induces $\pi_1(L_i)\rightarrow \pi_1(V)^{(n)}$, $\pi_1(D_i)\rightarrow\pi_1(V)^{(n)}$ (where $G^{(n)}$ denotes the $n^{\text{th}}$ term of the derived series).
\end{enumerate}
A knot is $n.5$-solvable if, in addition,
\begin{enumerate}[resume]
\item \label{prop:n5thterm}$\pi_1(L_i)\rightarrow\pi_1(V)^{(n+1)}$.
\end{enumerate}
$V$ is called the $n$-solution (respectively the $n.5$-solution) for $K$. The subset of $\mathcal{C}$ consisting of all $n$-solvable knots is denoted $\mf{F}_n$.
\end{definition}
In addition, if we employ some arbitrary commutator series $*$ on $\pi_1(V)$ and property \ref{prop:nthterm} (and \ref{prop:n5thterm}) holds for $\pi_1(V)^{(n)}_*$, we say that $K$ is $(n,*)$-solvable (respectively $(n.5,*)$-solvable).  The set of $(n, *)$-solvable knots is denoted by $\mf{F}_n^*$.  These definitions induce a filtration on the concordance group of knots indexed by half integers, where $\mf{F}_n \subset \mf{F}_n^*$ for each $n \in \frac{1}{2}\mathbb{Z}$ \cite[Proposition 2.5]{CHL5}.
\[
	0 \subset \bigcap \mf{F}^*_n \subset \dots \subset \mf{F}^*_{n.5}\subset\mf{F}^*_n\subset\dots\subset \mf{F}^*_1\subset\mf{F}^*_{0.5}\subset\mf{F}^*_0\subset\mathcal{C}
\]
There has been much work using such filtrations (see \cite{COT,CHL5,CHL6, CHL3}). Knots which are $0$-solvable are precisely those which have Arf-invariant zero, $1$-solvable knots are algebraically slice, and any topologically slice knot is in $\mf{F}_n$ for every $n\in \mathbb{Z}$.  Our results are based upon methods used in \cite{CHL5, CHL6, CHL3}.

We suppose $\mathfrak{R}$ is a knot and $\eta_1$, and $\eta_2$ are infecting curves in $S^3 - \mathfrak{R}$.  Let $J$ and $L$ be two knots which may or may not be distinct.   Then via \emph{infection}, we form $K_1$ by removing a tubular neighborhood, $\nu(\ea)$, of $\eta_1$ in $M_{\mathfrak{R}}$ and replace it with a copy of $S^3 - J$ along an identification of their common toric boundary.  This process is done such that the longitude of $J$ is identified with the meridian of $\nu(\eta_1)$ and the meridian of $J$ is identified with the reverse of the longitude of $\nu(\eta_1)$.  We denote this operation by $K_1 \equiv \rib(\ea,J)$.  Form $K_2$ similarly by infecting $\rib$ along $\eb$ with $L$.  Note that infection is really just a specified satellite operation.

\begin{theorem}\cite[Proposition 2.7]{CHL5}\label{thm:solvup}
Suppose $J \in \mf{F}^*_n$, $\rib$ is ribbon, and $\eta \subset S^3 - \rib$ is an infecting curve.  If $[\eta] \in \pi_1(M_\rib)_*^{(k)}$, then $\rib(\eta,J)$ is $(n+k, *)$-solvable.
\end{theorem}

Under certain conditions for the $\eta_i$, $J$, and $L$, we show that $K_1$ and $K_2$ are not concordant.  This is done by showing $\csum$ is not smoothly slice, as considered in the $(n,*)$-solvable filtration.  Both $J$ and $L$ will be $1$-solvable, and by Theorem \ref{thm:solvup}, both $K_i$ will lie in $\mf{F}_2$.  We show that $\csum$ is not slice by showing that it is not $(2.5, \mathcal{S})$-solvable where $\mathcal{S}$ is a commutator series defined in Definition \ref{def:Sseries}.

\subsection{Cheeger-Gromov constants, and the von Neumann $\rho$-invariant}
In the definition of an $n$-solvable knot, we rely heavily on properties of the $n$-solution $V$.  Therefore, we must look to invariants associated to this $4$-manifold in order to obstruct $n.5$-solvability.

Given a compact, orientable $4$-manifold $X$ with boundary $\partial X=M_K$, let $\Phi: \pi_1(X) \rightarrow \Lambda$ be a homomorphism where $\Lambda$ is a poly-torsion free abelian (PTFA) group \cite[Definition 2.5]{C},.  If $ \partial (X, \Phi) = (M_K, \phi)$, Cheeger and Gromov study the $\rho$-invariant, denoted $\rho(M_K, \phi)$, associated to this coefficient system and show that it is equal to the ``von Neumann signature defect" \cite{ChGr1}.
\[
	\rho(M_K, \phi) = \sigma_{\Lambda}^{(2)}(X,\Phi) - \sigma(X)
\]
In this equation, $\sigma_{\Lambda}^{(2)}(X, \Phi)$ is the $L^{(2)}$ signature of the equivariant intersection form defined on $H_2(X;\mathbb{Z}\Lambda)$ twisted by $\Phi$  and $\sigma(X)$ is the ordinary signature (See \cite[Section 5]{COT}).

\begin{proposition}\label{prop:rho}\cite[Proposition 4.1]{CHL6}
	\begin{enumerate}
		\item \label{prop:rhofactors} If $\phi$ factors through $\phi':\pi_1(M_K) \rightarrow \Lambda'$ where $\Lambda'$ is a subgroup of $\Lambda$, then $\rho(M_K, \phi') = \rho(M_K, \phi)$.
		\item \label{prop:rhotrivcomp} If $\Phi$ is trivial on the restriction to $M_K \subset \partial X$, then $\rho(M_K, \phi) = 0$.
		\item \label{prop:rhonaut} If $\phi:\pi_1(M_K) \rightarrow \mathbb{Z}$ is the abelianization homomorphism, then $\rho(M_K, \phi)$ is denoted by $\rho_0(K)$ and is equal to the integral of the Levine-Tristram signature function of $K$.
		\item \label{prop:rhoadd}The von Neumann signature defect satisfies Novikov additivity, i.e. if $X_1$ and $X_2$ intersect along a common boundary component and $\Phi_i$ is the restriction of $\Phi: X_1 \cup X_2 \rightarrow \Lambda$ to $X_i$, then $\sigma_{\Lambda}^{(2)}(X_1 \cup X_2, \Phi) = \sigma_{\Lambda}^{(2)}(X_1, \Phi_1) + \sigma_{\Lambda}^{(2)}(X_2, \Phi_2)$.
		\item \label{prop:chgr}There is a positive real number $C_K$ called the Cheeger-Gromov constant of $M_K$ such that, for any $\phi:\pi_1(M_K)\rightarrow \Lambda$, $|\rho(M, \phi)| < C_K$.
		\item \label{prop:rhosoln} Let $*$ be an arbitrary commutator series and $K \in \mf{F}_{n.5}^{*}$ via $X$ with $G=\pi_1(X)$.  If $\Phi: \pi_1(M) \rightarrow G/G_*^{(n+1)} = \Lambda$, then
			\[
				\sigma_{\Lambda}^{(2)}(X, \Phi) - \sigma(X) = 0 = \rho(M, \Phi)
			\]
	\end{enumerate}
\end{proposition}

Property \ref{prop:rhosoln} is integral to providing obstructions to solvability.  If we assume $V$ is a $(2.5)$-solution for $\csum$, and $\Phi:\pi_1(V)\rightarrow \Lambda$ is trivial on $\pi_1(V)^{(3)}$, then $\rho(M_{\csum},\phi)$ is trivial.

Using properties of the $\rho$-invariants, we make the choice of $J$ explicit. First, let $J_0$ be an Arf-invariant zero knot.  Take $R$ to be a ribbon knot with Alexander polynomial, $\Delta_R(t)\neq 1$, and an infecting curve, $\beta$ in $S^3- R$, which generates the rational Alexander module of $R$.  An example of such a choice is shown in Figure \ref{fig:Rk}, where the $k$ in the box denotes $k$ negative full twists, and $\Delta_R(t) = \left(kt-(k+1)\right)\left((k+1)t-k)\right)$.  We will require that $J_0$ have $|\rho_0(J_0)| > C_{R} + 2C_{\rib}$ where $C_R$ and $C_\rib$ are the Cheeger-Gromov constants of $R$ and $\rib$ respectively (properties \ref{prop:rhonaut} and \ref{prop:chgr} of Proposition \ref{prop:rho}).  By Theorem \ref{thm:solvup}, $J \in \mf{F}_1$.

\begin{figure}
\includegraphics[scale = .65]{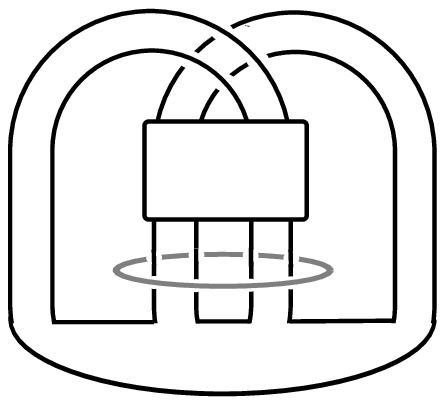}
		\put(-45,50){$k$}
		\put(-70,38){$\beta$}
	\caption{}\label{fig:Rk}
\end{figure}

\subsection{The Alexander Module and Blanchfield Form}
For any knot $K$ with Alexander polynomial $\Delta_K(t)\neq 1$, let $d$ be the leading coefficient of $\Delta_K(t)$ and $\mathcal{Q} = \Z[1/d]$.  The \emph{Alexander module of} $K$ \emph{with} $\mathcal{Q}$\emph{-coefficients} is defined by
\[
	\mathcal{A}^{\mathcal{Q}}(K) \equiv H_1(M_K; \mathcal{Q}[t, t ^{-1}]) 
		\cong\mathcal{A}^{\Z}(K) \otimes_\Z \mathcal{Q}.
\]
As a $\mathcal{Q}$-module, the Alexander module with  $\mathcal{Q}$ coefficients is finitely generated and free, i.e. $\mathcal{A}^{\mathcal{Q}}(K) \cong \mathcal{Q}^{2g}$ where $g$ is the $3$-genus of $K$.  Thus, any element $\gamma \in \mathcal{A}^{\mathcal{Q}}(K)$ may be described as a vector $(\gamma_1,\dots, \gamma_{2g}) \in \mathcal{Q}^{2g}$.

Classically the Blanchfield form, $\bl^\Z_K(\cdot,\cdot)$, is a sesquilinear form on the integral Alexander module of $K$.  This form has a generalization to the Alexander module with $\mathcal{Q}$-coefficients.

\begin{theorem}\cite[Theorem 2.13]{COT}
If $\mathcal{Q}$ is any ring such that $\mathbb{Z}\subseteq \mathcal{Q}\subseteq\mathbb{Q}$, then there is a nonsingular symmetric linking form
\[
	\bl^\mathcal{Q}_K: \mathcal{A}^{\mathcal{Q}}(K)\times\mathcal{Q}^{\mathcal{R}}(K)\rightarrow \mathbb{Q}(t) \mod \mathcal{Q}[t,t^{-1}].
\]
\end{theorem}

We will employ the Blanchfield form with $\mathcal{Q}$-coefficients for arbitrary $\Z\subseteq \mathcal{Q}\subseteq \mathbb{Q}$ and frequently alternate between coefficient systems.  As we are primarily concerned with instances such that $\bl^{\mathcal{Q}}(x,x)\neq \bl^{\mathcal{Q}}(y,y)$, the distinction is actually unnecessary for our purposes.  Suppose $K$ is a knot with Alexander polynomial $\Delta_K(t)$ and $x$ and $y$ are infecting curves in $S^3-K$.  Also, let $x$ and $y$ denote the corresponding elements in $\mathcal{A}^\mathbb{Z}(K)$.  Then $x\otimes1, y\otimes1$ are the images of $x,y$ under the map
\begin{equation}\label{eqn:modhomo}
	\mathcal{A}^\mathbb{Z}(K)\rightarrow \mathcal{A}^\mathcal{Q}(K)\cong \mathcal{A}^{\mathbb{Z}}(K)\otimes_{\mathbb{Z}}\mathcal{Q}
\end{equation}
given by $z\mapsto z\otimes 1$.  Since $\mathcal{A}^\mathbb{Z}(K)$ has no $\mathbb{Z}$-torsion, this map is injective.  The following proposition, though easy to show, was not found in the literature.  We prove it here for clarity.

\begin{proposition}\label{prop:blqz} For any ring $\mathcal{Q}$ such that $\mathbb{Z}\subseteq\mathcal{Q}\subseteq\Q$,
\[ \bl^\mathbb{Z}(x,x)=\bl^\mathbb{Z}(y,y) \qquad \iff \qquad \bl^\mathcal{Q}(x\otimes 1,x\otimes 1)=\bl^\mathcal{Q}(y\otimes 1,y\otimes 1)\]
\end{proposition}
\begin{proof}
The $\implies$ direction is obvious.  We prove the $\Longleftarrow$ direction by contrapositive.
Suppose $\bl^\mathbb{Z}(x,x) \neq \bl^\mathbb{Z}(y,y)$.  Then 
\[
	\bl^\Z(x,x)-\bl^\Z(y,y) = \frac{p(t)}{\delta_K(t)} \in \mathbb{Q}(t)\mod\mathbb{Z}[t,t^{-1}]
\]
where $(p(t),\delta_K(t))=1$ and $\delta_K(t)$ divides $\Delta_K(t)$.  Notice that the field of fractions of both $\Z[t,t^{-1}]$ and $\mathcal{Q}[t,t^{-1}]$ is $\Q(t)$ and the ring monomorphism $\Z\hookrightarrow\mathcal{Q}$ induces the following $\Z[t,t^{-1}]$-module homomorphisms. 
\begin{equation*}
\begin{array}{rl}
	h:& \Q(t)\hookrightarrow \Q(t)  \\
	\overline{h}:& \Q(t) / \Z[t,t^{-1}] \rightarrow \Q(t)/\mathcal{Q}[t,t^{-1}] \\
	h_\ast:& \mathcal{A}^\mathbb{Z}(K)\rightarrow \mathcal{A}^\mathcal{Q}(K)\cong \mathcal{A}^{\mathbb{Z}}(K)\otimes_{\mathbb{Z}}\mathcal{Q}
\end{array}
\end{equation*}
The first map is the identity and the third is equivalent to the map of Equation \ref{eqn:modhomo}.
However, given any element $z\in\mathcal{A}^\Z(K)$, we have
\begin{equation*}
	\bl^\mathcal{Q}(z\otimes 1,z\otimes 1) = \overline{h}(\bl^\Z(z,z))
\end{equation*}
\cite[Theorem 4.7]{Lei1}.  If $\bl^\mathcal{Q}(x\otimes 1,x\otimes 1)-\bl^\mathcal{Q}(y\otimes 1,y\otimes 1)=0$, this implies
\begin{equation*}
	\overline{h}\left(\frac{p(t)}{\delta_K(t)}\right)=0.
\end{equation*}
The map $\overline{h}$ is given by modding out by the subring  $\mathcal{Q}[t,t^{-1}]/\Z[t,t^{-1}]\subset \Q(t)/\Z[t,t^{-1}]$. This means $\frac{p(t)}{\delta_K(t)}$ reduces to a polynomial $F(t)\in\mathcal{Q}[t,t^{-1}]$.  After multiplying through by some constant $q\in \Z$ which is a unit in $\mathcal{Q}$, we obtain the following equation in $\Z[t,t^{-1}]$:
\begin{equation*}
	q\cdot p(t) = f(t)\delta_K(t),
\end{equation*}
where $ q\cdot F(t) = f(t) \in \Z[t,t^{-1}]$.  Since $\delta_K(1)=\pm 1$, regarding $q$ as a constant polynomial in $\Z[t,t^{-1}]$, $(q,\delta_K(t))=1$, so $q$ divides $f(t)$.  This means 
\[
	\frac{p(t)}{\delta_K(t)} = \frac{f(t)}{q} \in \Z[t,t^{-1}],
\]
and so $\bl^\Z(x,x) - \bl^\Z(y,y) = 0$ in $\Q(t) / \Z[t,t^{-1}]$, a contradiction.
\end{proof}

Because of Proposition \ref{prop:blqz}, we are free to suppress the distinction between the integral and rational Blanchfield forms in comparing the Blanchfield self-linking of two infecting curves.  We will frequently pass between the two and, by an abuse of notation, allow $\bl(x,x)$ to identify both $\bl^\Z(x,x)$ and $\bl^\mathcal{Q}(x\otimes 1, x\otimes 1)$ where understood.

Recall that $\ea$ and $\eb$ are infecting curves for $\rib$, and $J$ is a $1$-solvable knot defined by $J=R(\beta, J_0)$ where $|\rho_0(J_0)|>C_R + 2C_\rib$.  Then, require $L$ to be any $1$-solvable knot with Alexander polynomial $\Delta_L(t)$ satisfying one of the two following conditions:
\begin{enumerate}
	\item \label{cond1} $\Delta_R$ and $\Delta_L$ are \textit{strongly coprime}, i.e.  $\Delta_R(t^n), \Delta_L(t^m)$ are relatively prime for every $n, m \in \mathbb{Z}$.
	\item \label{cond2} $\Delta_R(t^m)$ and $\Delta_L(t^n)$ have no common roots unless $n= \pm m$.
\end{enumerate}
Certainly (\ref{cond1}) implies (\ref{cond2}).  If (\ref{cond1}) holds, $K_1 \equiv \rib(\ea, J)$ and $K_2\equiv \rib(\eb, L)$ are distinct (and even linearly independent) in $\mathcal{C}$ by a generalization of Cochran-Harvey-Leidy \cite{CHL5}.  If \ref{cond2}, a secondary restriction will be given by the Blanchfield self-linking of the infecting curves $\ea,\eb$.    In particular, we need only require that $\bl(\ea,\ea)\neq \bl(\eb,\eb)$.

\section{The Main Theorems}\label{sec:main}
%:	Main Theorem
\begin{theorem}\label{thm:distinct} If $R$ and $\rib$ are ribbon knots, let $J_0$ be an Arf-invariant zero knot such that $|\rho_0(J_0)| > C_R+2C_{\rib}$.  Suppose $J\equiv R(\beta,J_0)$ where $\beta$ generates the rational Alexander module of $R$.  Then form $K_1\equiv \rib(\ea, J)$ where $\bl(\ea,\ea) \neq 0$ and $K_2\equiv \rib(\eb,L)$.  If $L$ is any $1$-solvable knot such that
\begin{enumerate}
	\item $\Delta_L(t)$ and $\Delta_R(t)$ are strongly coprime, or
	\item $\Delta_L(t^m)$ and $\Delta_R(t^n)$ share a common root only when $n=\pm m$ and $\bl(\ea,\ea)\neq\bl(\eb,\eb)$,
\end{enumerate}
Then $K_1$ and $K_2$ are distinct in $\mathcal{C}$.
\end{theorem}

Before describing an outline of this proof, we introduce the following corollaries which illustrate the impact of Theorem \ref{thm:distinct}.

\begin{corollary}\label{cor:distinct}
Suppose $J \defined R(\beta,J_0)$ where $J_0$ is an Arf-invariant zero knot, $R$ is the ribbon knot from Figure \ref{fig:Rk} with $\beta$ as shown.  Let $K_1 \equiv \rib(\ea,J)$ and $K_2 \equiv \rib(\eb,J)$.  If $|\rho_0(J_0)|>C_R+2C_\rib$ and $\bl_\rib(\ea,\ea)\neq\bl_\rib(\eb,\eb)$ then $K_1$ and $K_2$ are not concordant.
\end{corollary}

\begin{proof}[Proof that Theorem \ref{thm:distinct} implies Corollary \ref{cor:distinct}]
We assume without loss of generality $\bl(\ea,\ea)\neq 0$.  Since $\Delta_R(t)=\Delta_J(t) = (kt-(k+1))((k+1)t-k)$ has roots $\{\frac{k}{k+1},\frac{k+1}{k}\}$, $\Delta_R(t^m)$ and $\Delta_R(t^n)$ share no common roots unless $n=\pm m$.  The result follows from Theorem \ref{thm:distinct}.
\end{proof}

The above stress the distinction between any two infecting curves.  We next generalize these results to produce infinitely many distinct concordance classes.

\begin{corollary}\label{cor:intense}
Suppose $\rib$ is any knot with $\Delta_\rib\neq 1$.  Then there exists a (countably infinite) set of curves $\{\eta_i\}$ in $S^3-\rib$ which are unknotted in $S^3$ and have linking number $0$ with $\rib$, and also a knot $J$ such that each $K_i\equiv \rib(\eta_i,J)$ generates a distinct concordance class in $\mathcal{C}$.
\end{corollary}
\begin{proof}
In order to employ Corollary $\ref{cor:distinct}$, we must ensure the existence of an infinite family of curves $\eta_i$ which have distinct Blanchfield self-linking, i.e. $\bl(\eta_i,\eta_i)=\bl(\eta_j,\eta_j)$ only when $i=j$.  Since $\rib$ has nontrivial Alexander polynomial and the Blanchfield form on $\mathcal{A}^\Z(\rib)$ is nonsingular, there must exist some curve $\eta \subset S^3-J$ such that $\bl(\eta,\eta)\neq 0$.  We use the following proposition.
\begin{proposition}
Suppose $\eta \subset S^3-\rib$ is an unknotted curve in $S^3$ with $lk(\eta,\rib)= 0$ and $\bl(\eta,\eta)\neq 0$.  For each $i\in \Z_{\geq 0}$, set $\eta_i = i\eta \in \mathcal{A}^\Z(\rib)$.  Then $\bl(\eta_i,\eta_i)=\bl(\eta_j,\eta_j)$ only when $i=j$, and each $\eta_i$ is represented by an unknotted curve in $S^3-\rib$.
\end{proposition}
\begin{proof}
Suppose $\bl(\eta,\eta) = \frac{p(t)}{\delta_\rib(t)} \notin \Z[t,t^{-1}]$ such that $(p(t),\delta_\rib(t))=1$ and $\delta_\rib(t)$ divides $\Delta_\rib(t)$.  Then we have
\begin{equation*}
	\bl(\eta_i,\eta_i) = \bl(i\eta,i\eta) = i^2\bl(\eta,\eta) = i^2\frac{p(t)}{\delta_\rib(t)}
\end{equation*}
If $\bl(\eta_i,\eta_i) = \bl(\eta_j,\eta_j)$, this implies $(i^2-j^2)\bl(\eta,\eta)= f(t) \in \Z[t,t^{-1}]$.  We have the following equation
\begin{equation*}
	(i^2-j^2)p(t)=f(t)\delta_\rib(t)
\end{equation*}
where, since $\frac{p(t)}{\delta_\rib(t)}\neq 0$, we can assume that $i^2-j^2$ does not divide $f(t)$ over $\Z[t,t^{-1}]$.  Since $\delta_\rib(q)=\pm 1$, $i^2-j^2 \in \{0,\pm 1\}$.  If $i^2-j^2 = \pm 1$, this contradicts $\bl(\eta,\eta)\neq 0$.  As $i,j \geq 0$,  $i^2-j^2$ is zero only when $i=j$.
We must next show that each $\eta_i$ is unknotted in $S^3$.  But notice that the element $i\eta \in \mathcal{A}^\Z(\rib)$ is realized by the $(i,1)$-cable of $\eta$.  This completes the proof.
\end{proof}
By taking $J$ to be the knot given in the statement of Corollary \ref{cor:distinct}, we obtain a family of pairwise distinct concordance classes $\{K_i\equiv\rib(\eta_i,J)\}$ for $i\geq 0$.
\end{proof}

The following corollary illustrates how uncommon it is for two infecting curves, $\eta$ and $\gamma$ in $S^3 - \rib$ to yield concordant knots.  By viewing them as elements of $\mathcal{A}^\Q(K)\cong \Q^{2g}$, we get an approximate answer to this question by seeing that the set of infecting curves $\gamma$ which yield knots concordant to $K = \rib(\eta, J)$ must lie on a quadratic hypersurface in $\Q^{2g}$.

%:	Corollary: dense
\begin{proposition}\label{prop:dense}
Let $\rib$ be a ribbon knot with Alexander polynomial $\Delta_\rib\neq 1$ and $J \equiv R(\beta,J_0)$ as above.  Fix some infecting curve $\eta \subset S^3 - \rib$ and let $K \equiv \rib(\eta,J)$.  Then,
\[
	\{ [\gamma] | \bl(\gamma,\gamma) = \bl(\eta,\eta) \}
\]
is the subset of a quadric hypersurface in $\mathbb{Q}^{2g}$, and thus $\{ [\gamma] | K' = \rib(\gamma, J_0) \text{ is not concordant to } K \}$ is dense as a subset of $\Q^{2g}$.
\end{proposition}

\begin{proof}
Following work of Trotter \cite{Tr0, Tr1}, let $z = (1-t)^{-1}$ and note that $\Q(t) = \Q(z)$.  Furthermore, since $z$ gives an automorphism of $\mathcal{A}^{\Z}(K)$, enlarging coefficients from $\Z[t,t^{-1}]$ to $\Z[t,t^{-1},z]$ has no effect on the module structure.  Consider the map
\[
	\frac{\Q(t)}{\Z[t,t^{-1}]}\xrightarrow{j} \frac{\Q(t)}{\Z[t,t^{-1},z]}
\]
given by inclusion.  The form given by $\widehat{\bl}(x,y) = j(\bl(x,y))$ is a nonsingular sesquilinear form and $j$ maps the image of $\bl(\cdot,\cdot)$ one-to-one onto the image of $\widehat{\bl}(\cdot,\cdot)$ \cite{Tr0}.

Using a partial fraction decomposition, any element in $\Q(t)$ may be written uniquely as the sum of a polynomial and proper fractions where the numerator has lower degree than the denominator.  Thus, $\Q(t)$ splits over $\Q$ as the direct sum of $\Q[t.t^{-1},z]$ and a subspace $P$ where $P$ consists of $0$ and proper fractions with denominator coprime to $t$ and $1-t$.  Then we have a $\Q$-linear map $\chi: \Q(t) \rightarrow \Q$ defined by 
\begin{equation*}
	\chi(f) = \left\{
			 \begin{array}{ll}
			f'(1) & f\in P \\
			0  & f\in \Q[t,t^{-1},z]
			\end{array} 
		\right.
\end{equation*}
Since $\chi$ is $0$ on $\Q[t,t^{-1},z]$, it is well defined on $\Q(z) \mod \Q[t,t^{-1},z]$ and thus on the image of $\widehat{\bl}$.  Note that the value of $\widehat{\bl}(x,y)$ is uniquely determined by the value of $\chi(\lambda \widehat{\bl}(x,y))$ for all $\lambda\in \Z[t,t^{-1},z]$, and furthermore, $\chi$ satisfies
\[
	\chi(\overline{f}) = \chi(f) \qquad \chi((t-1)f) = f(1)
\]
for any $f\in P$ \cite[Section 2]{Tr1}.  Since $\bl(x,y) = \overline{\bl(y,x)}$ for any $x,y \in \mathcal{A}^\Z(K)$, $\chi(\widehat{\bl}(\gamma,\gamma))=0$ for all $\gamma$.  This is also seen by noting that, by definition, $\bl$ is given by
\[
	\bl(x,y) = \overline{y}(1-t)\left(tV - V^\intercal\right)^{-1}x.
\]
Since $\bl$ is nonsingular, there must exist some $\lambda_0\in \Z[t,t^{-1},z]$ such that $\chi(\lambda_0\bl(x,x))$ is nonzero for some $x\in \mathcal{A}^{Z}(K)$.  For $\gamma \in S^3-K$, define $\widehat{\chi} : \mathcal{A}^\Q(K) \cong \Q^{2g} \rightarrow \Q$ by
\begin{align*}
	\widehat{\chi}(\gamma) &\equiv \chi((\lambda_0\widehat{\bl}(\gamma,\gamma)).
\end{align*}
Suppose $\widehat{\chi}(\eta) = c\in\Q$.  Then $\widehat{\chi}(x_1,\dots, x_{2g})=c$ is a rational equation in $2g$ variables and the left-hand side is a homogeneous polynomial of degree $2$.  That is,
\begin{align*}
	\widehat{\chi}(x_1,\dots,x_{2g}) =& \sum_{i,j} \chi\left(\lambda_0\bl(x_i,x_j)\right) \\
	=& \sum_{i,j}x_i,x_j \chi\left(\lambda_0\bl(e_i,e_j)\right)\\
	=& \sum_{i,j}a_{i,j} x_ix_j = c
\end{align*}
where $a_{i,j} = \chi\left(\lambda_0\bl(e_i,e_j)\right)$ and $\{e_i\}$ is a basis for $\mathcal{A}^\Q(K) \cong \Q^{2g}$.  Since $\bl$ is nonsingular, not all $a_{i,j}=0$.  By Theorem \ref{thm:distinct}, the set of infecting curves $\gamma\subset S^3-\rib$ such that $K' = \rib(\gamma,J)$ is concordant to $K = \rib(\eta,J)$ must be ones such that $\bl(\gamma,\gamma)= \bl(\eta,\eta)$.   Therefore $\gamma = (\gamma_1,\dots, \gamma_{2g})$ must be a solution to $\widehat{\chi}(x_1,\dots,x_{2g}) = c$.  

Consider the polynomial $F(x_1,\dots, x_{2g}) =\widehat{\chi}(x_1,\dots,x_{2g})-c =0$.  If $c\neq 0$, this polynomial is clearly nonconstant.  Otherwise, given the choice of $\lambda_0$ and that $\bl$ is nonsingular there must exist some element $\gamma\in \mathcal{A}^\Z$ such that $\widehat{\chi}(\gamma)\neq 0$ and hence $F(\gamma)\neq 0$ and $F$ is a nonconstant polynomial.  The zero locus of $\widehat{\chi}(x_1,\dots,x_{2g})-c$ is a quadric hypersurface in $\Q^{2g}$ whose compliment is dense.
\end{proof}

In the proof Proposition \ref{prop:dense}, we distinguish infecting curves by evaluating Trotter's trace function $\chi$ on $\lambda_0\widehat{\bl}(\gamma,\gamma)$ for one particular value $\lambda_0 \in \Q[t,t^{-1},z]$.  Since $\widehat{\bl}(\gamma,\gamma)$ is uniquely determined by the value of $\chi(\lambda \widehat{\bl}(\gamma,\gamma))$ for all $\lambda \in \Q[t,t^{-1},z]$, one could attempt to distinguish the infecting curves $\gamma$ and $\eta$ by using multiple values of $\lambda$ when $\chi(\lambda_0\widehat\bl(\gamma,\gamma))=\chi(\lambda_0\widehat\bl(\eta,\eta))$.

We now proceed to the proof of Theorem \ref{thm:distinct}.

%:	Proof
\begin{proof}[Proof of Theorem \ref{thm:distinct}]  We will show the stronger fact that $\csum$ is not $2.5$-solvable.  The proof is by contradiction.  $\csum$ is $2$-solvable by \cite[Proposition 2.7]{CHL5}, and suppose it is $2.5$-solvable via $V$.  We construct a tower of cobordisms for $M_{\csum}$.  Note that from the infection operation arises a natural cobordism between zero surgeries on the knots involved.  Given that $K_1\equiv \rib(\ea, J)$, denote by $F_1$ the cobordism obtained by first taking the disjoint union of  $M_{\rib}\times[0,1]$ and $M_{J}\times[0,1]$.  Then identify a neighborhood of $\ea \times \{1\}$, denoted by $\nu(\ea)$, in $M_{\rib}\times\{1\}$ with $\nu(J)$, a neighborhood of $J\times\{1\}$ in $M_{J}\times\{1\}$ given by $\left(M_J\setminus (S^3-J)\right)\times\{1\}$ as shown in Figure \ref{fig:F1}.  This identification is done such that the longitude of $\nu(J)$ is identified with the meridian of $\nu(\eta_1)$ and the meridian of $\nu(J)$ is identified with the reverse of the longitude of $\nu(\eta_1)$.  That is,
\begin{figure}
\includegraphics[scale=.7]{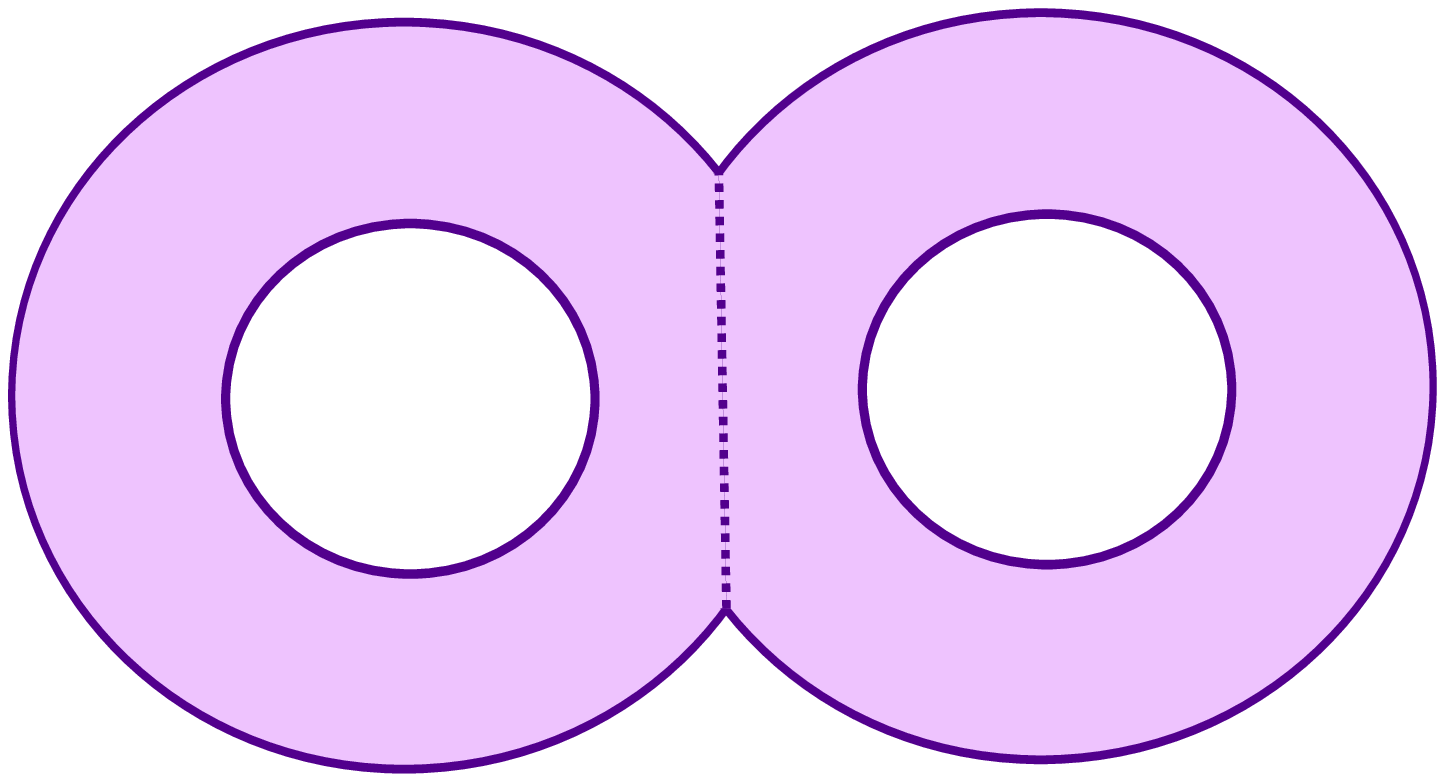}
\put(-165,110){$\nu(\eta)$}
\put(-100,25){$M_J\times [0,1]$}
\put(-235,25){$M_\rib\times [0,1]$}
\caption{$F_1$: Cobordism given by infection $K_1 =\rib(\ea,J)$}\label{fig:F1}
\end{figure}

\begin{equation*}
	F_1 \equiv \frac{(M_{\rib}\times[0,1])\cup(M_J\times[0,1])}{\nu(\ea) \sim \nu(J)}
\end{equation*}

\noindent
The boundary of $F_1$ is then given by $\bdy F_1 = M_\rib \sqcup M_J \sqcup \ov{M_{K_1}}$, where by $\ov{X}$, we mean the manifold $X$ with opposite orientation. Similarly, we let $F_2$ denote a cobordism given by the infections $K_2=\rib(\eb,L)$. The infection $J \equiv R(\beta,J_0)$ will yield a cobordism denoted $G$.  Since connected sum $\csum$ may also be viewed as the infection of $K_1$ by $-K_2$ along a meridian, form a cobordism $E$ between zero surgeries on $K_1$, $-K_2$, and $\csum$ in a similar manner.  Define $W'$ to be the union of $V$ and $E$ along their common boundary.  Similarly, $W$ is the union $W' \cup F_1\cup \ov{F_2}$.  Then, let $Z$ be the manifold obtained by joining the cobordisms $G$ to $W$ along $M_J$.  The boundary of $Z$ is given by $\partial Z = M_\rib \sqcup M_R \sqcup M_{J_0} \sqcup\ov{M_\rib} \sqcup \ov{M_L}$.  In overview,

\begin{align*}
	\bdy V &= M_{\csum} \\
	\bdy E &= M_{K_1} \sqcup M_{K_2} \sqcup \ov{M_{\csum}} \\
	\bdy F_1 &= M_{J} \sqcup M_\rib \sqcup \ov{M_{K_1}} \\
	\bdy F_2 &= M_{L} \sqcup M_{\rib} \sqcup \ov{M_{K_2}} \\
	\bdy G &= M_{J_0} \sqcup M_R \sqcup \ov{M_J}
\end{align*}

\begin{align*}
	W' &= V \cup_{M_{\csum}} E \\
	W	&= W' \cup_{M_{K_1}} F_1 \cup_{\ov{M_{K_2}}} \ov{F_2} \\
	Z	&= W \cup_{M_J} G.
\end{align*}
				                    
\begin{figure}
	\begin{center}
	\includegraphics[scale =1]{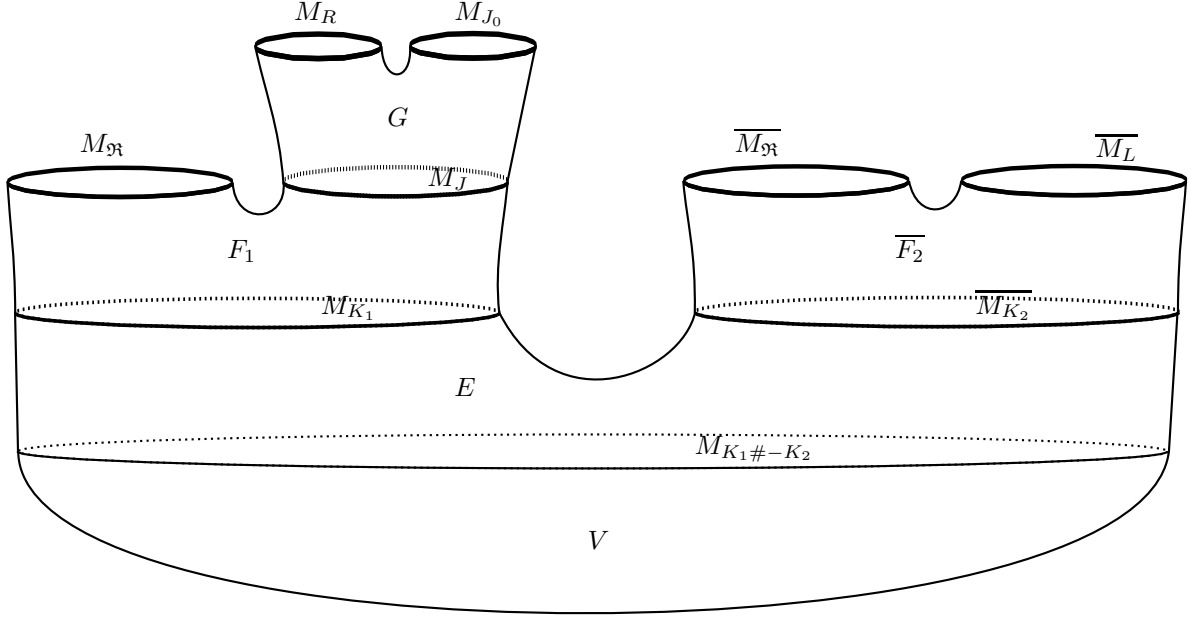}
		\put(-225, 25){\textbf{$V$}}\put(-185,61){$M_{\csum}$}
		\put(-275,83){\textbf{$E$}}\put(-80,114){$\overline{M_{K_	2}}$}\put(-325,114){$M_{K_1}$}			\put(-360,135){\textbf{$F_1$}}\put(-285,162){$M_J$}\put(-415,175){$M_\rib$}
		\put(-110,135){\textbf{$\ov{F_2}$}}\put(-35,173){$\ov{M_L}$}\put(-170,175){$\ov{M_\rib}$}
		\put(-300,185){\textbf{$G$}}\put(-335,225){$M_{R}$}\put(-275,225){$M_{J_0}$}
	\caption{The $4$-manifold $Z$, constructed by a tower of cobordisms}\label{tower}
	\end{center}
\end{figure}
 
Unfortunately, the derived series itself will not be useful in finding an obstruction to the $2.5$-solvablity of $\csum$.  Instead, we define a partial commutator series, $\mathcal{S}$, which will be slightly larger than the rational derived series so that
\[
	\pi_1(Z)^{(3)} \subset \pi_1(Z)_{\mathcal{S}}^{(3)}.
\]
Notice in Definition \ref{def:Sseries}, $\mathcal{S}$ will be equivalent to the rational derived series on its first two terms.

%:		Commutators defined.
\begin{definition}\label{def:Sseries}
Let $G$ be a group with $G/G^{(1)} = \langle \mu \rangle \cong \mathbb{Z}$, then the \textit{derived series localized at $\mathcal{S}$} is defined recursively by
\begin{equation*}
	\begin{split}
		G^{(0)}_{\mathcal{S}} &\equiv G \\
		G^{(1)}_{\mathcal{S}}  = G_r^{(1)} &\equiv \ker \left(G \rightarrow \frac{G}{[G,G]}\otimes_{\mathbb{Z}}\mathbb{Q}\right) \\
		G^{(2)}_{\mathcal{S}} = G_r^{(2)} & \equiv \ker \left(G_{\mathcal{S}}^{(1)} \rightarrow 
			\frac{G_{\mathcal{S}}^{(1)}}{[G_{\mathcal{S}}^{(1)},G_{\mathcal{S}}^{(1)}]}\otimes_{\mathbb{Z}[G/G_{\mathcal{S}}^{(1)}]}\mathbb{Q}[G/G_{\mathcal{S}}^{(1)}]\right)\\
		G_{\mathcal{S}}^{(3)}  &\equiv \ker \left( G_{\mathcal{S}}^{(2)} \rightarrow 
			\frac{G_{\mathcal{S}}^{(2)}}{[G_{\mathcal{S}}^{(2)},G_{\mathcal{S}}^{(2)}]}\otimes_{\mathbb{Z}[G/G_{\mathcal{S}}^{(2)}]}\mathbb{Q}[G/G_{\mathcal{S}}^{(2)}]S^{-1}\right).
	\end{split}
\end{equation*}
\end{definition}
\noindent
The right divisor set $ S ~\subset ~\mathbb{Q}[G_{\mathcal{S}}^{(1)}/G^{(2)}_{\mathcal{S}}] \subset \mathbb{Q}[G/G^{(2)}_{\mathcal{S}}]$ is the multiplicative set generated by  $\{ \Delta_L(\mu^i \eta_2' \mu^{-i}) | i \in \mathbb{Z} \}$.  Here, ${\eta_2}'$ denotes the image of $\eb$ in $M_{-\rib}\subset M_{\rcsum}$ and is considered as an element of $\pi_1(W)$ by inclusion.  $S$ is by definition a multiplicatively closed set with unity, and $0$ is not an element of $\mathcal{S}$.  Since $\mathbb{Q}[G^{(1)}/G_{\mathcal{S}}^{(2)}]$ is abelian, this verifies $S$ is a right divisor set.  Furthermore, let $\gamma \in G/G_{\mathcal{S}}^{(2)}$.  If $q(a) \in S$, then $\gamma q(a) \gamma^{-1} = q(\gamma a \gamma^{-1}) \in S$. Therefore, $\mu^{i}\eb'\mu^{-i}$ is invariant under conjugation by $G/G_{\mathcal{S}}^{(2)}$, and we see that $S$ is invariant under conjugation by $\mathbb{Q}[G/G_{\mathcal{S}}^{(2)}]$.  

Consider the coefficient system on  $W$ given by the projection
\begin{equation*}
	\Phi : \pi_1(Z) \rightarrow \pi_1(Z)/\pi_1(Z)^{(3)} \rightarrow \pi_1(Z)/\pi_1(Z)^{(3)}_{\mathcal{S}}\equiv \Lambda
\end{equation*}
Because of property (\ref{prop:rhoadd}) of Proposition \ref{prop:rho} (and after suppressing notation by $\sigma_{\Lambda}^{(2)} \equiv \sigma^{(2)} \text{ and } \Phi |_{X} \equiv \Phi$ where understood), we have

\begin{equation}\label{eqn:sigdef4}
	\begin{split}
	\sigdef{Z}{\Phi}  =	&    \left(\sigdef{V}{\Phi}\right) + \left(\sigdef{E}{\Phi}\right)  + \left(\sigdef{F_1}{\Phi}\right) \\
				  	&+  \left(\sigdef{\ov{F_2}}{\Phi}\right) + \left(\sigdef{G}{\Phi}\right)
	\end{split}
\end{equation}

\noindent By assumption, $V$ is a $2.5$-solution.  Property (\ref{prop:rhosoln}) of Proposition \ref{prop:rho} yields $\sigdef{V}{\Phi} = 0$.  For $E, F_1, F_2,$ and $G$, all of the (integral and twisted) second homology comes from the boundary \cite[Lemma 2.4]{CHL3}, and therefore

\[
	\sigdef{E}{\Phi} = \sigdef{F_1}{\Phi} = \sigdef{F_2}{\Phi}=\sigdef{G}{\Phi} = 0.
\]

\noindent However, $\sigdef{Z}{\Phi} = \rho(\partial Z, \Phi |_{\partial})$, and 

\[
	0 	= \rho(\partial Z, \Phi)  
		= \rho(M_{J_0},\Phi) + \rho(\ov{M_{L}}, \Phi) + \rho(M_\rib,\Phi)+\rho(\ov{M_\rib},\Phi)+ \rho(M_R,\Phi).
\]

We employ the following lemmas, to be proven in Section \ref{sec:bforms}.

%:		Lemmas
\medskip
\noindent{\bf Lemma \ref{lemma:rhoJ}}\quad
{\sl The restriction of $\Phi$ to $\pi_1(M_{J_0})$ factors non-trivially through $\mathbb{Z}$.}
\\

\noindent
{\bf Lemma \ref{lemma:rhoL}}\quad
{\sl The restriction of $\Phi$ to $\pi_1(M_L)$ also factors through $\Z$ and yields $\rho(M_L,\Phi) =0.$}\\

\noindent After proving Lemma \ref{lemma:rhoJ} and using properties (\ref{prop:rhofactors}) and (\ref{prop:rhonaut}) of Proposition \ref{prop:rho}, we will have $\rho(M_{J_0}, \Phi) = \rho_0(J_0)$.  Secondly, by Lemma \ref{lemma:rhoL} and property (\ref{prop:rhotrivcomp}) of Proposition \ref{prop:rho}, $\rho(\ov{M_L},\Phi) = -\rho(M_L,\Phi) = 0$.  This yields the following equation.

\[
	\rho_0(J_0) = - \rho(M_{\rib}, \Phi) - \rho(M_{-\rib},\Phi)  - \rho(M_{R},\Phi).
\]

\noindent 	This is a contradiction since, by hypothesis,
\[
	|\rho_0(J_0)| > C_R+2C_\rib\geq  \rho(M_{\rib}, \Phi)+  \rho(M_{-\rib},\Phi)+  \rho(M_{R},\Phi)
\]
This completes the proof modulo the proofs of Lemmas \ref{lemma:rhoJ} and \ref{lemma:rhoL}
\end{proof}

%%%%%%%%%%%%%%%%%%%%%%%%%%%%%
%											%
%		Section 5: Nontriviality and Triviality			%
%											%
%%%%%%%%%%%%%%%%%%%%%%%%%%%%%

\section{Blanchfield Form Restrictions}\label{sec:bforms}
In this section, we prove the Lemmas needed for the completion of Theorem \ref{thm:distinct}.  We continue to use notation which was defined in Sections \ref{sec:background} and \ref{sec:main}.  Before proving Lemma \ref{lemma:rhoJ}, we must first show that the infecting curve $\ea$ represents a nontrivial element of $\pi_1(W)^{(1)}/\pi_1(W)_\mathcal{S}^{(2)}$ by inclusion.  Note that $\pi_1(M_{J_0})$ is normally generated by the meridian $\mu_0$ which is isotopic in $Z$ to the $\beta \in \pi_1(M_R)^{(1)}$.  Similarly, the meridian of $M_R$ is identified with $\ea$ and inclusion induces
\begin{equation*}
	\ea\in\pi_1(M_\rib)^{(1)} \subset \pi_1(W)^{(1)} \subset \pi_1(Z)^{(1)}
\end{equation*}
which implies that $\mu_0 \sim \beta$ is in $\pi_1(Z)^{(2)}$.  If $\ea \in \pi_1(Z)^{(2)}$, then $\pi_1(M_{J_0})$ is mapped to a subset of $\pi(Z)^{(3)}$ and the restriction of $\Phi$ to $\pi_1(M_{J_0})$ is trivial.

Continue to let $\eta_2' \subset M_{\rcsum}$ denote the image of $\eta_2$ after reversing the orientation of $M_\rib$ and taking the connected sum to form $M_{\rcsum}$.  By an abuse of notation, $\ea$ and $ \eta_2'$ also represent the corresponding elements in the Alexander module and $\pi_1$.  Let $\mathcal{A}(X)$ denote the Alexander module of the space $X$ with rational coefficients.  The following proofs closely follow the methodology of \cite[Lemmas 7.6, 7.6]{CHL6}

%:	Lemma Nontrivial in pi^1/pi^2
\begin{lemma}\label{lemma:nontrivial}
The infecting curve $\ea$ represents a nontrivial element of $A \equiv \pi_1(W)^{(1)}/\pi_1(W)^{(2)}_\mathcal{S}$
\end{lemma}

\begin{proof} 
Consider the following commutative diagram of Alexander modules.  

\begin{equation}\label{giant_alexander_diagram}
	\begin{diagram}\dgARROWLENGTH=1.0em
		\node{\mathcal{A}^{\mathbb{Z}}(\mathfrak{R}\#- \mathfrak{R})} \arrow{s,l}{i_1} \arrow{e,t}{\phi_*}
		\node{\mathcal{A}^{\mathbb{Z}}(V)} \arrow{s,l}{i_2} \arrow{e,t}{f_*}
		\node{\alex{W'}}\arrow{s,l}{i_3} \arrow{e,t}{g_*}
		\node{\mathcal{A}^{\mathbb{Z}}({W})} \arrow{s,l}{i_5} \arrow{e,t}{}
		\node{A} \arrow{sw,l}{i_6} \\
		\node{\mathcal{A}(\mathfrak{R}\#- \mathfrak{R})} \arrow{e,t}{\phi'_*}
		\node{\mathcal{A}(V)} \arrow{e,t}{f'_*}
		\node{\qalex{W'}}\arrow{e,t}{g'_*}
		\node{\mathcal{A}(W)}
	\end{diagram}
\end{equation}

\noindent The validity of this diagram is supported by the fact that $\alex{\csum} \cong \alex{\rcsum}$.  The horizontal maps are induced by inclusion.  Since $\alex{\rcsum}$ is $\mathbb{Z}$ torsion free, $i_1$ is injective.  By Definition \ref{def:Sseries}, $\pi_1(W)^{(2)}_{\mathcal{S}} \equiv \pi_1(W)^{(2)}_r$, and therefore $i_6: \pi_1(W)^{(1)}/\pi_1(W)^{(2)}_{\mathcal{S}} \rightarrow \mathcal{A}(W)$ is clearly injective.

The kernel of $\phi'_*$ is an isotropic submodule of $\mathcal{A}(\mathfrak{R}\#-\mathfrak{R})$ with respect to the Blanchfield form. Since the rational Alexander module of $\rcsum$ decomposes under connected sum, as does its Blanchfield form, $\ea$ must be mapped to a nontrivial element of $\mathcal{A}(V)$ as $\bl_\rib(\eta_1,\eta_1) \neq 0$.

It remains to show that the lower maps $f'_*, g'_*$ are injective; that is, the rational Alexander module of $V$ injects into that of $W$.  Since the connected sum operation may be described as an infection $\csum \equiv K_1(\mu_1,-K_2)$, the kernel of $f'_*:\pi_1(M_{\csum}) = \pi_1(\partial V) \rightarrow \pi_1(E)$ is normally generated by the longitude of $-K_2$ as a curve in $\pi_1(M_{K_1})$  \cite[Lemma 2.5(1)]{CHL3}.  The longitude lies in the second derived subgroup of $\pi_1(K_2)$ and also in the second derived subroup of $\pi_1(M_{\csum})$.  Since the rational Alexander module of a space, $X$, with $H_1(X) \cong \mathbb{Z}$ is given by $\mathcal{A}(X) \cong \pi_1(X)^{(1)} / \pi_1(X)^{(2)} \otimes_{\mathbb{Z}} \mathbb{Q}$, $f_*$ is an isomorphism between the rational Alexander modules of $V$ and $W'$. 

Similarly, to show $g'_*$ is injective, we note that this kernel is normally generated by the longitudes of $J$ and $L$ as curves in $M_{K_1}$ and $\ov{M_{K_2}}$ respectively.  These curves lie in $\pi_1(M_J)^{(2)}$ and $\pi_1(\ov{M_L})^{(2)}$, contained via inclusion in $\pi_1(M_{K_1})^{(3)}$ and $\pi_1(\ov{M_{K_2}})^{(3)}$ respectively, and $g_*'$ is an isomorphism.
\end{proof}

For the contradiction used in the proof of Theorem \ref{thm:distinct} we show that  $\mu_0 \sim \beta$ is nontrivial as an element of $\pi_1(Z)^{(2)}/\pi_1(Z)^{(3)}_{\mathcal{S}}$.  

%: 	Lemma: rho(J,Phi)
\begin{lemma}\label{lemma:rhoJ}
The meridian of $J_0$, $\mu_0$, isotopic in $Z$ to $\beta$, is nontrivial as an element of
\[
\frac{\pi_1(Z)^{(2)}}{\pi_1(Z)^{(3)}_{\mathcal{S}}}
\]
Therefore, the restriction $\Phi: \pi_1(M_{J_0}) \rightarrow  \pi_1(Z)/\pi_1(Z)^{(3)}_{\mathcal{S}} = \Lambda$ factors nontrivially through $\mathbb{Z}$.
\end{lemma}

\begin{proof}
Recall that the kernel of 

\[
	\pi_1(W)\rightarrow \pi_1(W \cup G) = \pi_1(Z)
\]

\noindent
is the normal closure in $\pi_1(W)$ of the kernel of $\pi_1(M_J) \rightarrow \pi_1(G)$.  This is normally generated by the longitude of the infecting knot $J_0$ considered as a curve in $S^3- J_0 \subset M_J \subset \partial W$ \cite[Lemma 2.5 (1)]{CHL3} which lies in $\pi_1(M_{J_0})^{(2)}$.  Inclusion induces
\begin{equation*}
\pi_1(M_{J_0})^{(2)} \subset \pi_1(M_J)^{(3)} \subset \pi_1(W)^{(3)} \subset \pi_1(W)^{(3)}_{\mathcal{S}}
\end{equation*}
as well as the following isomorphism:
\begin{equation*}
	\frac{\pi_1(W)}{\pi_1(W)^{(3)}_{\mathcal{S}}} \cong \frac{\pi_1(Z)}{\pi_1(Z)^{(3)}_{\mathcal{S}}} = \Lambda
\end{equation*}

\noindent
Therefore, it suffices to show $\beta$ is nontrivial ${\pi_1(W)}/{\pi_1(W)^{(3)}_{\mathcal{S}}}$.  Consider the following commutative diagram, where we set $\Gamma \equiv \pi_1(W)/\pi_1(W)^{(2)}_\mathcal{S}$ and $\mathcal{R} \equiv \mathbb{Q}\Gamma S^{-1}$.

\begin{equation}
	\begin{diagram}\label{diag:ea_in_pis}\dgARROWLENGTH=1.0em
		\node{\pi_1(M_J)^{(1)}} \arrow{s,l}{} \arrow[2]{e,t}{j_*}
		\node[2]{\pi_1(W)^{(2)}} \arrow{s,l}{} \arrow{e,t}{\Phi}
		\node{\frac{\pi_1(W)^{(2)}_{\mathcal{S}}}{\pi_1(W)^{(3)}_{\mathcal{S}}}} \arrow{s,l}{j} \\
		\node{\mathcal{A}(J)\otimes \mathcal{R}} \arrow{e,t}{\cong}
		\node{H_1(M_J;\mathcal{R})} \arrow{r,t}{j_*}
		\node{H_1(W; \mathcal{R})} \arrow{e,t}{\cong}
		\node{\frac{\pi_1(W)^{(2)}_{\mathcal{S}}}{[\pi_1(W)^{(2)}_{\mathcal{S}}, \pi_1(W)^{(2)}_{\mathcal{S}}]}\otimes \mathcal{R}}
	\end{diagram}
\end{equation}

\noindent
We will now justify certain maps of the diagram.  Here, the horizontan map $j_*$ is given by functoriality of the commutator series and inclusion which induces $\pi_1(M_J) \subset \pi_1(W)^{(1)}$.  Since $\pi_1(M_J)$ is normally generated by the meridian $\mu_1$ which is identified with $\ea$ in $W$ and $\ea$ is nontrivial in $A = \pi_1(W)^{(1)}/\pi_1(W)^{(2)}_{\mathcal{S}}$ by Lemma \ref{lemma:nontrivial}, the map
\[
	\pi_1(M_J) \rightarrow
	\frac{\pi_1(W)^{(1)}}{\pi_1(W)^{(2)}_{\mathcal{S}}} \hookrightarrow
	\frac{\pi_1(W)}{\pi_1(W)^{(2)}_{\mathcal{S}}} \equiv \Gamma
\]
must factor nontrivially through $\pi_1(M_J)/\pi_1(M_J)^{(1)} = \langle \mu_1 \rangle \cong\mathbb{Z}$.  It follows that 

\[
	H_1(M_J;\mathbb{Q}\Gamma) \cong H_1(M_J;\mathbb{Q}[t, t^{-1}])\otimes \mathbb{Q}\Gamma
	\equiv \mathcal{A}(J)\otimes_{\mathbb{Q}[t, t^{-1}]}\mathbb{Q}\Gamma
\]

\noindent
where $\mathbb{Q}[t, t^{-1}]$ acts on $\mathbb{Q}\Gamma$ by $t\mapsto \ea$.  Thus, $H_1(M_J;\mathcal{R}) \cong \mathcal{A}(J)\otimes \mathcal{R}$.  To justify the map 

\begin{equation}\label{eqn:h1wriso}
	H_1(W;\mathcal{R}) \xrightarrow{\cong} 
	\frac{\pi_1(W)^{(2)}_\mathcal{S}}{[\pi_1(W)^{(2)}_{\mathcal{S}}, \pi_1(W)^{(2)}_{\mathcal{S}}]}\otimes\mathcal{R},
\end{equation}

\noindent
note that we may interpret $H_1(W;\mathbb{Z}\Gamma)$ as the first homology of the $\Gamma$ covering space of $W$, so

\[
	H_1(W;\mathbb{Z}\Gamma)\xrightarrow{\cong}
	\frac{\pi_1(W)^{(2)}_\mathcal{S}}{[\pi_1(W)^{(2)}_{\mathcal{S}}, \pi_1(W)^{(2)}_{\mathcal{S}}]}.
\]

\noindent
Since $\mathcal{R}$ is a flat $\mathbb{Z}\Gamma$-module, equation (\ref{eqn:h1wriso}) is justified.  Moreover, by the definition of $\pi_1(W)^{(3)}_\mathcal{S}$ in Definition \ref{def:Sseries}, the vertical map $j$ is injective.  Recall that by hypothesis, $\beta$ generates the rational Alexander module of $R$, and hence $J$, which implies $\beta\otimes 1$ is the generator of $H_1(M_J;\mathcal{R}$.  Therefore, in order to finish the proof, it suffices to show that $\beta \otimes 1$ is not in the kernel of the bottom row of (\ref{diag:ea_in_pis}).

Note that $W$ is given by $V\cup E\cup F_1 \cup \ov{F_2}$ with $\partial W = M_\rib \sqcup M_J \sqcup \ov{M_\rib}\sqcup \ov{M_L}$.  Since $E, F_1, F_2$ have no second homology relative boundary, 
\begin{equation*}
	 \frac{H_2(W)}{i_*\left(H_2(\partial W)\right)} \cong H_2(V).
\end{equation*}
Furthermore, $V$ is a $2$-solution and therefore $H_2(W)/i_*\left(H_2(\partial W)\right)$ has a basis which satisfies conditions \ref{prop:basis} and \ref{prop:nthterm} of Definition \ref{def:nsolvable} though it fails condition \ref{prop:inclusion}.  Therefore, $W$ is called a \textit{$2$-bordism} for $\partial W$ \cite[Definition 7.11]{CHL5}.

Suppose $P \equiv \ker \{j_*:H_1(M_J;{\mathcal{R}})\rightarrow H_1(W;\mathcal{R})\}$. Then, since $W$ is a $2$-bordism, by \cite[Theorem 7.15]{CHL5}, $P$ must be an isotropic submodule of $H_1(M_J;{\mathcal{R}})$ with respect to the Blanchfield form $H_1(\partial W;\mathcal{R})$ and thus on $H_1(M_J;\mathcal{R})$.  However, we have already shown that $\beta\otimes 1$ is a generator of $H_1(M_J;{\mathcal{R}})$, and if $\beta\otimes 1 \in P \equiv \ker j_*$, then $\bl^{\mathcal{R}}_J(\beta\otimes 1, \beta \otimes 1) =0$.  Since $\bl^{\mathcal{R}}$ is nonsingular \cite[Lemma 7.16]{CHL5}, this means $H_1(M_J;\mathcal{R}) \equiv 0$.  In order to give a contradiction, we show

\[
	\mathcal{A}(J)\otimes \mathcal{R} \cong 
	\left( \frac{\mathbb{Q}\Gamma}{\Delta_R(\ea)\mathbb{Q}\Gamma}\right) S^{-1} \neq 0.
\]

%: 		Q Alex of R_i nontrivial
 By hypothesis of Theorem \ref{thm:distinct}, the rational Alexander module of $R$ is nontrivial, and $\Delta_R(t)$ is not a unit in $\mathbb{Q}[t, t^{-1}]$.  The map $\mathbb{Z}\rightarrow \Gamma$ given by $t \mapsto \ea$ is nontrivial, since we showed in Lemma \ref{lemma:nontrivial} that $\ea \neq 0$ in $\pi_1(W)^{(1)}/\pi_1(W)^{(2)}_{\mathcal{S}}$.  Since $\Gamma$ is torsion-free, $\mathbb{Q}\Gamma$ is a free left $\mathbb{Q}[\ea,\ea^{-1}]$-module on the right cosets of $\langle \ea \rangle \subset \Gamma$, where $\langle\ea\rangle$ denotes the submodule of $\mathbb{Q}\Gamma$ generated by $\ea$.  We may then fix a set of coset representatives so that any $x \in \mathbb{Q}\Gamma$ has a unique decomposition 

\[
	x = \sum_{\xi} x_{\xi}\xi
\]

\noindent
where each $x_{\xi} \in \mathbb{Q}[\ea,\ea^{-1}]$ and each $\xi$ is a coset representative in $\Gamma$.  Notice that if $\Delta_R(\ea)x=1$ then 

\[
	\Delta_R(\ea)x = \Delta_R(\ea) \sum_{\xi}x_{\xi}\xi = \sum_{\xi}\Delta_R(\ea)x_{\xi}\xi = 1.
\]

\noindent
This implies that on the coset $\xi = e$, we have $\Delta_R(\ea)x_e = 1$ in $\mathbb{Q}[\ea, \ea^{-1}]$, contradicting the fact that $\Delta_R(t)$ is not a unit in $\mathbb{Q}[t, t^{-1}]$.  Therefore, $\Delta_R(\ea)$ has no right inverse in $\mathbb{Q}\Gamma$.  Since $\Gamma$ is poly-torsion-free abelian, $\mathbb{Q}\Gamma$ is a domain \cite{Str} and 
\[
	\frac{\mathbb{Q}\Gamma}{\Delta_R(\ea)\mathbb{Q}\Gamma} \ncong 0.
\]

\noindent
Next, we consider the localization of this module at $S$. The kernel of 

\[
	\frac{\mathbb{Q}\Gamma}{\Delta_R(\ea)\mathbb{Q}\Gamma}\rightarrow
	\frac{\mathbb{Q}\Gamma}{\Delta_R(\ea)\mathbb{Q}\Gamma} S^{-1}
\]

\noindent
is the $S$-torsion submodule \cite[Cor 3.3, p 57]{Ste}.  So to establish the desired result, it suffices to show that the generator of $\mathbb{Q}\Gamma/\Delta_R(\ea)\mathbb{Q}\Gamma$ is not $S$-torsion.  If this generator, which we denote by $1$, is $S$-torsion, then $1s = \Delta_R(\ea)y$ for some $s\in S$ and $y\in \mathbb{Q}\Gamma$.

Remember that $\Gamma \equiv \pi_1(W)/\pi_1(W)^{(2)}_{\mathcal{S}}$ and $A \equiv \pi_1(W)^{(1)}/\pi_1(W)^{(2)}_{\mathcal{S}} \lhd \Gamma$.  Since $A \subset \Gamma$, we may view $\mathbb{Q}\Gamma$ as a free left $\mathbb{Q}A$-module on the set of right cosets of $A$ in $\Gamma$.  So any $y \in \mathbb{Q}\Gamma$ has a unique decomposition 

\[
	y = \sum_{\xi} y_{\xi}\xi,
\]

\noindent
where the sum is over a set of coset representatives $\{\xi \in \Gamma\}$ and $y_\xi$ is an element of $\mathbb{Q}A$.  Then

\[
\begin{split}
	s = &\Delta_R(\ea) y\\
	   = &\Delta_R(\ea) \sum_{\xi}y_{\xi}\xi.
\end{split}
\]

\noindent
Since $s \in S \subset \mathbb{Q}A$ and $\Delta_R(\ea)\in \Q A$, it must be that each coset representative $\xi \neq e$ yields $0=\Delta_R(\ea)y_{\xi}$.  Note that $\mathbb{Q}[\ea,\ea^{-1}] \subset \mathbb{Q}\Gamma$ and hence $\Delta_R(\ea)\neq 0$.  Since $\Q A\subset \Q \Gamma$ is a domain, it must be that $y_\xi=0$ for all $\xi\neq e$.  Therefore $y \in \mathbb{Q}A$ and $s = \Delta_R(\ea)y$ is an equation in $\mathbb{Q}A$.  Because of Definition \ref{def:Sseries}, each element of $S$ can be written as the product of terms of the form $\Delta_L(\mu^i \eb' \mu^{-i})$.

Moreover, since $A$ is a torsion-free abelian group, we may view $s = \Delta_R(\ea)y$ as an equation in $\mathbb{Q}F$ for some free abelian group $F \subset A$ of finite rank $r$.  Since $\mathbb{Q}F$ is a UFD, we apply the following proposition.

%: 		Proposition: Strongly coprime
\begin{proposition}\cite[Proposition 4.5]{CHL5}\label{prop:coprimes} Suppose $\Delta_R(t), \Delta_L(t) \in \mathbb{Q}[t, t^{-1}]$ are non zero.  Then $\Delta_R$ and $\Delta_l$ are strongly coprime if and only if, for any finitely generated free abelian group $F$ and any nontrivial $a,b \in F$, $\Delta_R(a)$ is relatively prime to $\Delta_L(b)$ in $\mathbb{Q}F$.
\end{proposition}

Recall if $s = \Delta_R(\ea)y$ is an equation in $S$, $\Delta_R(\ea)$ must divide a product of terms of the form $\Delta_L(\mu^i \eb \mu^{-i})$.  If $\Delta_R, \Delta_L$ are strongly coprime, we already arrive at a contradiction, since Proposition \ref{prop:coprimes} implies $\Delta_R(\ea)$ is relatively prime to $\Delta_L(\mu^i \eb' \mu^{-i})$ for any $i$.   Otherwise, choose some basis $\{ x_1, x_2, \dots, x_r\}$ for $F$ such that $\ea = x_1^m$ for some positive $m\in \mathbb{Z}$.  Then $\mu^i \eb' \mu^{-i} = x_1^{n_{i,1}}x_2^{n_{i,2}}\cdots x_r^{n_{i,r}}$, and we may view $\mathbb{Q}F$ as a Laurent Polynomial ring in the variables $\{x_1,x_2,\dots,x_r\}$.  Since $\Delta_R\neq 0$ and is not a unit, there exists some nonzero complex root, $\zeta$, of $\Delta_R(x_1^m)$.  Suppose that $\tilde{p}(x_1)$ is a nonzero irreducible factor of $\Delta_R(x_1^m)$ of which $\zeta$ is a root.  Then for some $i$, $\tilde{p}(x_1)$ divides $\Delta_L(x_1^{n_{i,1}}x_2^{n_{i,2}}\cdots x_r^{n_{i,r}})$ and so $\zeta$ must be a zero of $\Delta_L(x_1^{n_{i,1}}x_2^{n_{i,2}}\cdots x_r^{n_{i,r}})$ for every complex value of $x_2, \dots, x_r$ which is impossible unless $n_{i,j} = 0$ for each $j>1$.  Therefore, $\mu^i \eb' \mu^{-i} = x_1^{n_i}$ for some $n_i \neq 0$.  Recall that $\Delta_R(t^m)$ and $\Delta_L(t^n)$ share no common roots unless $n=\pm m$.  Thus $n_i = \pm m$ and $\mu^i\eb'\mu^{-i} = (\ea)^{\pm 1}$ for some $i$. 

This equation holds in $A$ but each of $\ea,\eb',$ and $\mu$ are given by circles in $M_{\rcsum}$ where $\mu^i\eb'\mu^{-1}$ and $\ea$ represent elements of $\mathcal{A}^\Z(\rcsum)$.  Therefore, the validity of the equation $\mu^i\eb'\mu^{-i} = (\ea)^{\pm 1}$ may be considered in $\mathcal{A}^{\mathbb{Z}}(\rcsum)$ as long as $(\mu^i \eb' \mu^{-i})  \ea^{\mp 1}$ does not lie in the kernel of 
\[
	\mathcal{A}^{\mathbb{Z}}(\rib\#-\rib) \rightarrow
	\mathcal{A}^{\mathbb{Z}}(W) \rightarrow
	\frac{\pi_1(W)^{(1)}}{\pi_1(W)^{(2)}_{\mathcal{S}}} \equiv A
\]
Notice however, that in the module notation for $\mathcal{A}^{\mathbb{Z}}(\rcsum)$, $(\mu^i \eb' \mu^{-i}) \ea^{\mp 1} = \tau_*^i(\eb')\pm\ea$, and we consult the Blanchfield form:
\begin{equation*}
\begin{array}{rl}
	\bl_{\rcsum}(\tau_*^i(\eb') \pm\ea, \tau_*^i(\eb') \pm\ea) =& \bl_{-\rib}(\tau_*^i(\eb'),\tau_*^i(\eb))\pm \bl_{\rib}(\ea,\ea) \\
	=& \bl_{-\rib}(\eb',\eb')+\bl_\rib(\ea,\ea) \\
	=& -\bl_\rib(\eb,\eb)+\bl_\rib(\ea,\ea) \\
	\neq & 0.
\end{array}
\end{equation*}
The last inequality holds since the requirement imposed upon $\ea,\eb$ was that $\bl(\ea,\ea) \neq \bl(\eb,\eb)$.  Therefore, if the equality $\mu^i\eb'\mu^{-i} = (\ea)^{\pm 1}$ holds in $A$, it must hold in $\mathcal{A}^{\mathbb{Z}}(\rib\#-\rib)$ where it is written as $\tau_*^i(\eb') = \ea^{\pm 1}$.  Let $U$ and $U'$ be Seifert matrices for $\rib$ and $-\rib$ respecitvely. We remark that although $U'=-U$,  this distinction is made to emphasize the different contributions from the respective basis elements coming from the Seifert surfaces of $\rib$ and $-\rib$.  A presentation matrix for the Alexander module $\mathcal{A}^{\mathbb{Z}}(\rcsum)$ is given by
\begin{equation*}
	\left(\begin{array}{cc}
		U-\tau_*U^{\intercal} & 0 \\
		0 & U' - \tau_*U'^{\intercal}  \end{array}
	\right)
\end{equation*}
The automorphism $\tau_*$ decomposes under connected sum $\rcsum$.  Thus $\tau_*(\alex{\mathcal{R}}\oplus 0) \subset \alex{\mathcal{R}} \oplus 0$ and $\tau_*(0\oplus \alex{\mathcal{-R}}) \subset 0 \oplus \alex{\mathcal{-R}}$, and invalidates the equation $ \tau_*^i(\eb') = \ea^{\pm 1}$.  This contradicts the equality of the statement $\mu^i\eb'\mu^{-i}=\ea^{\pm 1}$ in $A$ and, therefore, contradicts the assumption of $\Delta_R(\ea)$ being $S$ torsion.  Thus, $\mathcal{A}(J)\otimes \mathcal{R}$ is nontrivial and $\beta\otimes 1$ cannot lie in the kernel of the bottom row of \ref{diag:ea_in_pis}.  This completes the proof that $\mu_0 \sim \beta$ is nontrivial in $\pi_1(Z)^{(2)}_\S/\pi_1(Z)^{(3)}_\S$ so the restriction of $\Phi$ to $\pi_1(M_{J_0})$ factors nontrivially through $\Z$.
\end{proof}

Our last task is to show that $\rho(M_L,\Phi)=0$, completed in the following short lemma.

%:	Lemma: rho(L,Phi)
\begin{lemma}\label{lemma:rhoL}
The restriction of $\Phi$ to $\pi_1(\ov{M_L})$ also factors through $\Z$ and $\rho(M_L,\Phi) = 0$.
\end{lemma}

\begin{proof}
Similar to the beginning of Lemma \ref{lemma:rhoJ}, we begin with the following commutative diagram. 

\begin{equation}
	\begin{diagram}\label{diag:L}\dgARROWLENGTH=1.0em
		\node{\pi_1(\ov{M_L})^{(1)}} \arrow{s,l}{} \arrow[2]{e,t}{j_*}
		\node[2]{\pi_1(W)^{(2)}} \arrow{s,l}{} \arrow{e,t}{\Phi}
		\node{\frac{\pi_1(W)^{(2)}_{\mathcal{S}}}{\pi_1(W)^{(3)}_{\mathcal{S}}}} \arrow{s,l}{j} \\
		\node{\mathcal{A}(\ov{L})\otimes \mathcal{R}} \arrow{e,t}{\cong}
		\node{H_1(\ov{M_L};\mathcal{R})} \arrow{r,t}{j_*}
		\node{H_1(W; \mathcal{R})} \arrow{e,t}{\cong}
		\node{\frac{\pi_1(W)^{(2)}_{\mathcal{S}}}{[\pi_1(W)^{(2)}_{\mathcal{S}}, \pi_1(W)^{(2)}_{\mathcal{S}}]}\otimes \mathcal{R}}
	\end{diagram}
\end{equation}

Again, $j_*$ is given by functoriality of the comutator series and inclusion given that $\pi_1(\ov{M_L}) \subset \pi_1(W)^{(1)}$.  $\pi_1(\ov{M_L})$ is normally generated by its meridian, $\mu_{\ov{L}}$, which is identified with $\eb'$.   Suppose that $\eb'$ is nontrivial in $\pi_1(W)^{(1)}/\pi_1(W)^{(2)}_{\mathcal{S}}$.

\begin{equation*}
	\pi_1(\ov{M_L}) \rightarrow \frac{\pi_1(W)^{(1)}}{\pi_1(W)^{(2)}_{\mathcal{S}}} \hookrightarrow \frac{\pi_1(W)}{\pi_1(W)^{(2)}_{\mathcal{S}}}.
\end{equation*}

\noindent
This map must factor through $\pi_1(\ov{M_L})/\pi_1(\ov{M_L})^{(1)} = \langle \mu_L \rangle \cong \mathbb{Z}$, and

\begin{equation*}
	H_1(\ov{M_L};\mathbb{Q}\Gamma) \cong H_1(\ov{M_L};\mathbb{Q}[t,t^{-1}])\otimes \mathbb{Q}\Gamma \equiv \mathcal{A}(L)\otimes_{\mathbb{Q}[t, t^{-1}]}\mathbb{Q}\Gamma.
\end{equation*}

\noindent
where $\mathbb{Q}[t,t^{-1}]$ acts on $\mathbb{Q}\Gamma$ by $t \mapsto \mu_L \simeq \eb'$.  Therefore, $H_1(M_L;\mathcal{R})\cong\mathcal{A}(L)\otimes\mathcal{R}$.  Since the rational Alexander module of $L$ is $\Delta_L(t)$-torsion and $\Delta_L(\eb') \in S$ by definition,  this module is trivial.  This implies that the map along the top row of Diagram \ref{diag:L} is zero.  

Conversely, suppose $\eb'$ is trivial in $\pi_1(W)^{(1)}/\pi_1(W)^{(2)}_{\mathcal{S}}$.  Since $\pi_1(M_L)$ is normally generated by $\mu_L \simeq \eb'$ in $Z$, this implies $j_*(\pi_1(M_L))\subset \pi_1(W)^{(2)}_{\mathcal{S}}$ by inclusion and the map along the top row of the diagram is again zero.

Finally, consider the restriction of $\Phi$ to $\pi_1(\ov{M_L})$.

\begin{equation*}
	\Phi: \pi_1(\ov{M_L}) \rightarrow \frac{\pi_1(W)}{\pi_1(W)^{(3)}_{\mathcal{S}}}
\end{equation*}

\noindent
By the above arguments, this map is trivial on the subgroup $\pi_1(\ov{M_L})^{(1)} \subset \pi_1(\ov{M_L})$ and must therefore factor through $\pi_1(\ov{M_L})/\pi_1(\ov{M_L})^{(1)} \cong \mathbb{Z}$.  There are two easy cases to consider.  If the map is trivial, we have $\rho(M_L;\Phi) = 0$.  Otherwise, the map factors nontrivially through $\mathbb{Z}$ and $\rho(M_L;\Phi) = \rho^0(L)=0$ since $L$ is a 1-solvable knot.  This finishes the proof of the above lemma and completes the proof of Theorem \ref{thm:distinct}.

\end{proof}

%%%%%%%%%%%%%%%%%%%%%%%%%%%%%
%											%
%		Section ?: 9_46 example					%
%											%
%%%%%%%%%%%%%%%%%%%%%%%%%%%%%

\section{Example: $\mathfrak{R} = 9_{46}$}\label{sec:application1}
In this section, we give an explicit example of Corollary \ref{cor:intense} where we take $\mathfrak{R} = 9_{46}$ so that $\Delta_{\mathfrak{R}}(t) = -2t^2+5t-2$.  The infecting curves $a,b$, as shown in Figure \ref{fig:946_eta}, generate the integral Alexander module of $\mathfrak{R}$, and $\eta = a+b$ generates the rational Alexander module.  In $\mathcal{A}^\Z(\rib)$, we have the relations:
\begin{align}
	2ta = a \Rightarrow (2t-1)a = 0, \label{eqn:ainmod} \\
	tb = 2b \Rightarrow (t-2)b = 0. \label{eqn:binmod}
\end{align}
Any element, $\gamma$, of the integral Alexander module may be written as a polynomial combination of $a, b$, that is $\gamma = x(t)a+y(t)b \in \mathcal{A}^{\mathbb{Z}}$, where $x(t),y(t) \in \mathbb{Z}[t,t^{-1}]$.  Let $\mathcal{Q}$ denote the subring $\mathbb{Z}[2^{-1}] \subset \mathbb{Q}$.  Consider the map
\[
	\mathcal{A}^{\mathbb{Z}}(\rib) \rightarrow \mathcal{A}^{\mathbb{Z}}(\rib)\otimes_{\mathbb{Z}}\mathcal{Q}.
\]
Because of identities \ref{eqn:ainmod} and \ref{eqn:binmod},  
\[
	t^ra \mapsto 2^{-r}a, \qquad t^rb\mapsto 2^rb.
\]
Therefore, 
\[
	x(t)a \mapsto x(2^{-1})a \qquad y(t)b \mapsto y(2)b
\]
and $\gamma \mapsto x(2^{-1})a+y(2)b$, where $x(2^{-1}), y(2)  \in \mathcal{Q}\subset\mathbb{Q}$.  These equations hold as we map to the rational Alexander module.
\[
	\mathcal{A}^{\mathbb{Z}}(\rib) \rightarrow \mathcal{A}^{\mathbb{Z}}(\rib)\otimes_{\mathbb{Z}}\mathcal{Q} \rightarrow \mathcal{A}^{\mathbb{Z}}(\rib)\otimes_\mathbb{Z}\mathbb{Q} \equiv \mathcal{A}(\rib)
\]

\begin{figure}[h]
\includegraphics[scale=.3]{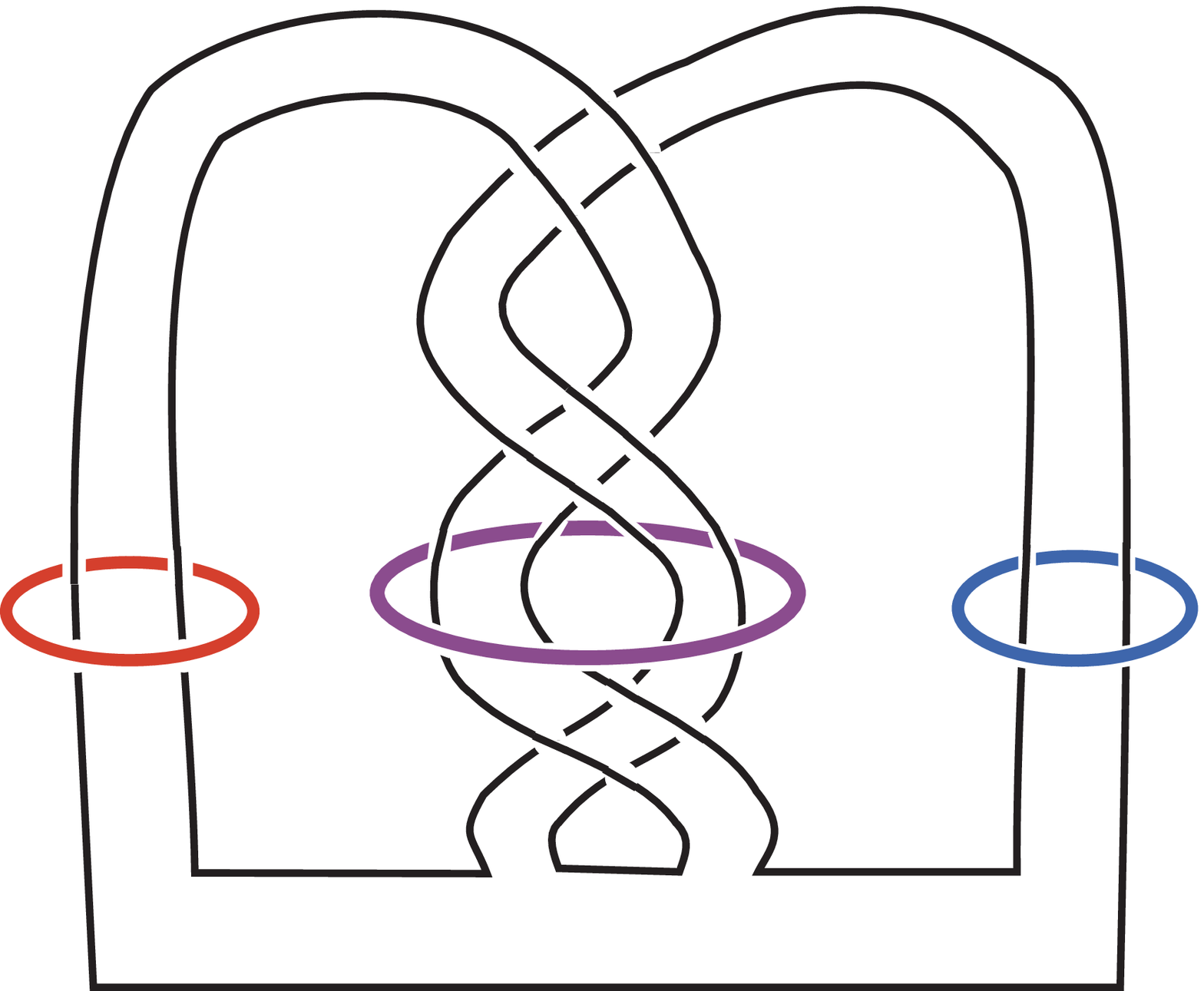}
\put(-155,45){\color{red}$a$}\put(3,45){\color{blue}$b$}\put(-50,55){\color{purple}$\eta$}\qquad
\caption{$\rib=9_{46}$,  $\eta=a+b$}\label{fig:946_eta}
\end{figure}

Suppose we fix $\eta = a+b$.  Let $K_1 = \mathfrak{R}(\eta; J)$ where $J$ is built as in the statement of Corollary \ref{cor:distinct}.  Suppose $\gamma = x(t)a+y(t)b \in \mathcal{A}^{\mathbb{Z}}(\rib)$, and let $K_2 = \rib(\gamma; J)$.  The rational Blanchfield self-linking of $\gamma$ is given by
\begin{equation}\label{eqn:946calcs}
\begin{array}{rl}
	\bl_\rib(\gamma,\gamma) =& \bl_\rib\big(\big(x(t)a+y(t)b\big),\big(x(t)a+y(t)b\big)\big)\\
		=&  \bl_\rib\left(\left(x\left(2^{-1}\right)a+y\left(2\right)b\right),\left(x\left(2^{-1}\right)a+y\left(2\right)b\right) \right)\\
		=& \left[x\left(2^{-1}\right)^2\bl_\rib(a,a)\right] + \left[x\left(2^{-1}\right)y\left(2\right)\bl_\rib(a,b)\right]\\
		 & +\left[x\left(2^{-1}\right)y(2)\bl_\rib(b,a)\right]+\left[y\left(2\right)^2\bl_\rib(b,b)\right] \\
		=& x\left(2^{-1}\right)y\left(2\right)\big(\bl_\rib(a,b)+\bl_\rib(b,a)\big) \\
		=& x\left(2^{-1}\right)y(2)\bl_\rib(\eta,\eta)
\end{array}
\end{equation}
where $\bl_\rib(a,a)=\bl_\rib(b,b)=0$ since $a$ and $b$ both generate isotropic submodules of $\mathcal{A}^\mathbb{Z}(\rib)$.  Corollary \ref{cor:distinct} states that $K_1$ and $K_2$ are distinct up to concordance as long as $\bl_\rib(\eta,\eta)\neq\bl_\rib(\gamma,\gamma)$ which from \ref{eqn:946calcs} is equivalent to $\left(1- x(2^{-1})y(2)\right)\bl_\rib(\eta,\eta)\neq 0$.  A formula for the Blanchfield form can be given by a Seifert matrix $U$ for $\rib$:
\begin{equation*}
	\bl(r,s) = \ov{s}(1-t)\left( tU-U^\intercal \right)^{-1}r
\end{equation*}
where $\ov{s}$ is the image of $s$ under the involution $t\mapsto t^{-1}$.  The Seifert matrix for $\rib$ yielding a presentation matrix for $\mathcal{A}(\rib)$ with respect to the basis $\{a,b\}$ is
\begin{equation*}
	\left(\begin{array}{cc}
		0	&	-1 \\
		-2	&	0 \\
	\end{array}\right),
\end{equation*}
and by a simple calculation,
\[
	\bl(\eta,\eta) = \frac{3(t-1)^2}{\Delta_{\rib}(t)}, \text{ where } (3(t-1)^2, \Delta_\rib(t))=1.
\]
This implies $(1-x(2^{-1})y(2))\bl_\rib(\eta,\eta)$ is zero if and only if $1-x(2^{-1})y(2)$ is a multiple of $\Delta_{\rib}(t)$.  This is only possible if $x(2^{-1})$ and $y(2)$ are inverses in ${\mathcal{Q}}\subset \mathbb{Q}$, and it must be that $x(2^{-1})=\pm2^{-r}$, $y(2) = \pm2^{r}$ with the same sign.  Therefore, $x(t)a$ and $y(t)b$ are equivalent in $\mathcal{A}^{\mathbb{Z}}(\rib)$ to $\pm t^{r}a$ and $\pm t^{r}b$ respectively and with the same sign.  Therefore, $x(t)a+y(t)b \equiv \pm ( t^ra + t^rb )= \pm t^r\eta$.  Since $\pm t^r \eta$ is represented by the infecting curve $\pm \eta$ in $S^3-\rib$, regardless of $r$, we see that infection upon $\eta$ and $\gamma$ may yield concordant knots only if $\gamma = \pm \eta$.

\begin{figure}[h]
	\includegraphics[scale=.25]{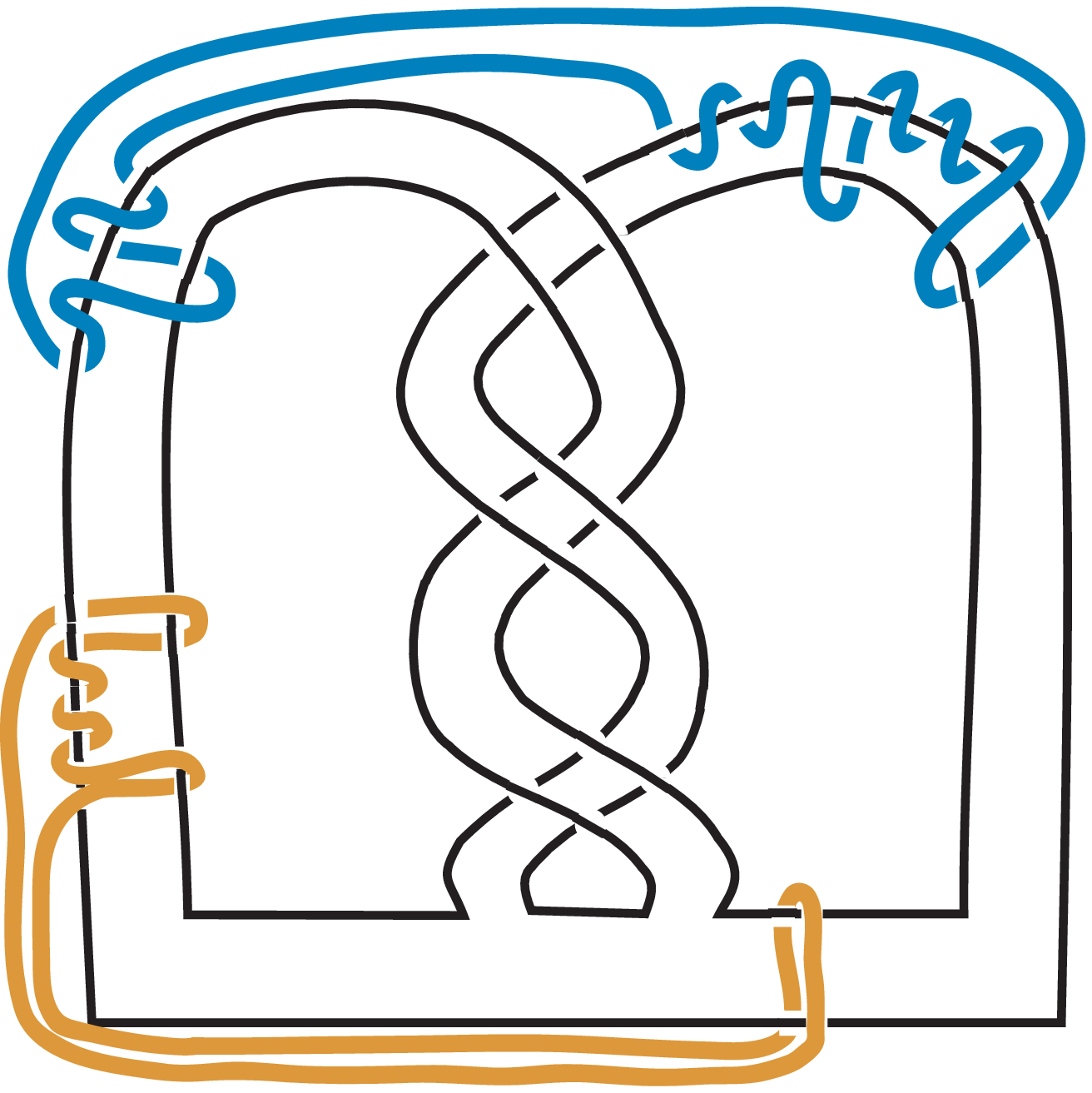}
	\put(-108,30){$\gamma_1$}
	\put(-108,80){$\gamma_2$}
	\caption{$\gamma_1 = (t+t^{-1})a+b$, $\gamma_2 = ta+(t^2+1)b$}\label{fig:52s}
\end{figure}

More generally, let $\gamma_i = x_i(t)a+y_i(t)b$ where $x_i(t), y_i(t) \in \mathbb{Z}[t,t^{-1}]$ for $i=1,2$.  Then by \ref{eqn:946calcs}, $\bl(\gamma_1,\gamma_1)=\bl(\gamma_2,\gamma_2)$ if and only if $(x_1y_1-x_2y_2)\bl(\eta,\eta)=0$, where for simplicity we set $x_i \equiv x_i(2^{-1})$ and $y_i \equiv y_i(2) \in \mathcal{Q} \subset \mathbb{Q}$. This is zero in $\mathbb{Q}(t)/\mathcal{Q}[t,t^{-1}]$ when $x_1y_1 = x_2y_2$ in $\mathcal{Q}$.  For every distinct value $c_i \in \mathbb{Z}[1/2]$, we can find an infecting curve $\gamma_i \subset S^3-\rib$ such that $\bl^\mathbb{Q}(\gamma_i,\gamma_i) = c_i\bl^\mathbb{Q}(\eta,\eta)$.  If $c_i=\hat{c_i}2^{-k_i}$ for $\hat{c_i}, k_i\in \mathbb{Z}$,  $\gamma_i$ may be given by $\gamma_i=t^{k_i}a+\hat{c_i}b$.  Thus, each $c_i$ yields a distinct concordance class $K_i\equiv \rib(\gamma_i;J)$.  We summarize these results in the following lemma and also in the graph of Figure \ref{fig:graaaph}.

\begin{lemma}
Let $\rib$ be the $9_{46}$ knot and $J$ the knot given in Corollary \ref{cor:distinct}.  For every $c_i \in \mathbb{Z}[1/2]$, we obtain an unknotted curve $\eta_i \subset S^3-J$ such that $lk(\eta_i,\rib)=0$.  By infection, the $\{\eta_i\}$ yield infinitely many distinct concordance classes of knots $K_i\equiv \rib(\eta_i,J)$.
\end{lemma}

\begin{figure}
\includegraphics[scale=.75]{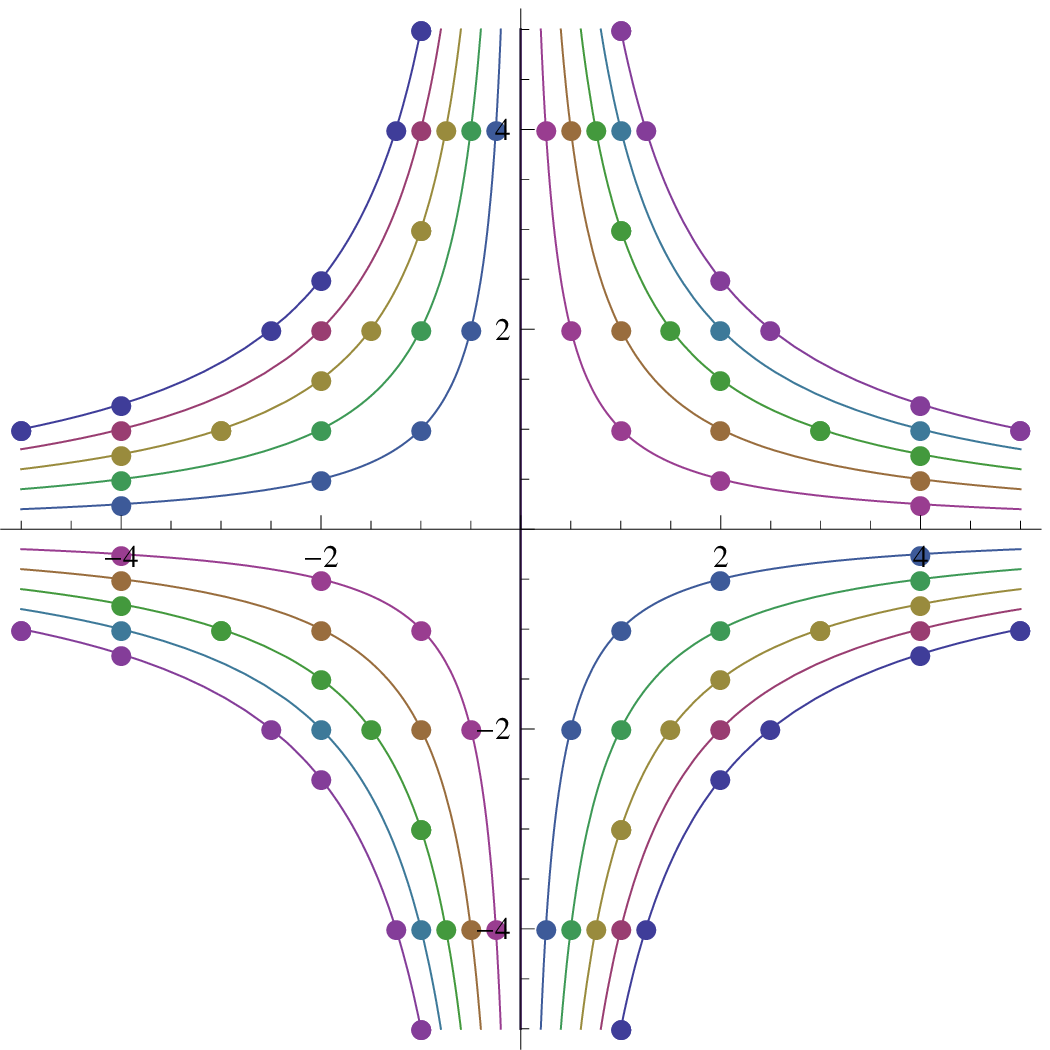}
	\caption{Each level curve in this graph is given by $xy=c\in\mathbb{Z}[1/2]$.  Actual $\gamma = xa+yb$ in $S^3-\rib$ with $\bl(\gamma,\gamma)=c$ are represented by shaded points $(x,y)$ on the level curves.  Choices of $\gamma_i$ lying on different level curves lead to nonconcordant knots $\rib(\gamma_i;J)$}
\label{fig:graaaph}
\end{figure}

Nonetheless, there are many combinations of $x_1, y_1, x_2, y_2 \in \mathbb{Z}[1/2]$ for which $x_1y_1 = x_2y_2$.  For instance, take $\gamma_1 = (t+t^{-1})a+b, \gamma_2 = ta+(t^2+1)b$ as in Figure \ref{fig:52s}. Although these curves are not isotopic in $S^3-\rib$, $x_1y_1=x_2y_2=15/4$ implying that $\gamma_1+\gamma_2'$ lies in an isotropic submodule of the rational Alexander module, $\mathcal{A}(\rcsum)$, and thus potentially in the kernel of the map
\[
	\mathcal{A}^{\mathbb{Z}}(\rcsum)\overset{\phi_*}\longrightarrow\mathcal{A}^{\mathbb{Z}}(V)
\]
for some potential $2.5$-solution, $V$, of $\csum$.  Infection upon $\eta_1$ and $\eta_2$ by $J$ may thus produce concordant knots as we saw in Example \ref{ex2}.

\section{When infecting curves have distinct orders in $\mathcal{A}(\rib)$}\label{sec:1solv}
In many cases, we do not need the full strength of Theorem \ref{thm:distinct} in order to find an obstruction to concordance between knots given by $\rib(\ea,J)$ and $\rib(\eb,J)$.  Many previous findings on the structure of the knot concordance group, and in particular the $n$-solvable filtration, have relied only on the order of the infecting curve $\eta_i$ as an element of the rational Alexander module \cite{CHL5,CHL6, CHL3}.  In this section, by building on these previous results, we show that when $\ea$ and $\eb$ have different orders as elements of $\mathcal{A}(\rib)$, distinct concordance classes are found more readily and with slightly weaker hypothesis.

Let $\rib$ again be any knot.  We will take $J_0$ to be some Arf invariant zero knot such that $|\rho_0(J_0)| > 2C_\rib$, where $C_\rib$ is the Cheeger-Gromov constant of $\rib$ and $L$ is any knot.  Let $\ea$ and $\eb$ be infecting curves whose orders as elements of $\alex{\rib}$ are $o_1(t)$ and $o_2(t)$, respectively.  Suppose $p(t)$ is a prime polynomial such that $p(t)$ divides $o_1(t)$ but $(o_2(t), p(t)) =(o_2(t),p(t^{-1}))=1$.  Then take $P$ to be the multiplicative subset of $\mathbb{Q}[t,t^{-1}]$ given by
\begin{equation*}
	P = \{ q_1(t)q_2(t)\cdots q_t(t) | (q_i,p) = (q_i,\ov{p}) = 1\},
\end{equation*}
and let $\Psi$ be the homomorphism $\Psi: \mathbb{Q}[t,t^{-1}] \rightarrow \mathbb{Q}[t,t^{-1}]P^{-1}$.  Then $\Psi$ induces the map (which, by abusing notation, we also label $\Psi$):
 \begin{equation*}
 	\Psi:\frac{{\mathbb{Q}(t)}}{{\mathbb{Q}[t,t^{-1}]}} \rightarrow \frac{{\mathbb{Q}(t)}}{\mathbb{Q}[t,t^{-1}]P^{-1}}.
\end{equation*}
In this case, we require $\Psi(\bl(\ea,\ea))\neq 0$, which yields the following theorem.

\begin{theorem}
Let $J_0$ and $\rib$ be knots such that $|\rho_0(J_0)| > 2C_\rib$. Let $\ea, \eb$ be infecting curves such that the orders of $[\ea],[\eb]$ in $\alex{\rib}$ are $o_1(t),o_2(t)$ respectively.  If there exists a prime $p(t)$ dividing $o_1(t)$  such that $(p,o_2)=(\ov{p},o_2)=1$ and $\Psi(\bl(\ea,\ea))\neq 0$, then given any knot $L$ which is $0$-solvable,  $K_1 = \rib(\ea,J)$ and $K_2 = \rib(\eb,L)$ are distinct in $\mathcal{C}$.
\end{theorem}

\begin{proof}
Construct $K_1$ and $K_2$ as indicated in the statement of the theorem.  Certainly, both $K_i$ are $1$-solvable.  Let $E_1$ and $E_2$ be the cobordisms given by the infections $K_1 \equiv \rib(\ea,J)$ and $K_2 \equiv \rib(\eb,L)$ respectively, and let $F$ be the cobordism given by the connected sum $K_1\#-K_2$.  As in the proof of Theorem \ref{thm:distinct}, we show by contradiction that $\csum$ is not slice.  If $\csum$ is slice, there exists a slice disk complement $V$ with boundary $\bdy V=M_{\csum}$.  Let $W$ be the manifold obtained by adjoining $V$ to $F$ along $M_{\csum}$ and similarly $Z$ is obtained by adjoining $W$ to $E_1$ and  $\ov{E_2}$ along $M_{K_1}$ and $\ov{M_{K_2}}$ respectively.  Then $\bdy Z = M_\rib \sqcup M_{J_0} \sqcup \ov{M_\rib} \sqcup \ov{M_L}$.

Take $\mathcal{P}$ to be a partial commutator series on the class of groups $G$ with $\beta_1 = 1$, given by
\begin{equation*}
\begin{array}{rl}
	G^{(0)}_\p&= G \\
	G^{(1)}_\p&= G^{(1)}_r \\
	G^{(2)}_\p&= \ker \{G^{(1)} \rightarrow \frac{G^{(1)}}{[G^{(1)},G^{(1)}]}\otimes_{\mathbb{Z}[t,t^{-1}]}\mathbb{Q}[t,t^{-1}]P^{-1}\}.
\end{array}
\end{equation*}
Let $\phi$ be the projection
\begin{equation*}
	\phi:\pi_1(Z)\rightarrow\frac{\pi_1(Z)}{\pi_1(Z)^{(2)}}\rightarrow\frac{\pi_1(Z)}{\pi_1(Z)^{(2)}_\p}
\end{equation*}
Again, we consider the von Neumann signature defect of $Z$ given by this coefficient system.
\begin{equation}\label{eqn:sigdeffinal}
\begin{array}{rl}
	0 =& \sigdef{Z}{\phi}  \\
	=& \rho(\bdy Z,\phi|_{\pi_1(\partial Z)}) \\
	=&\rho(M_\rib,\phi|_{\pi_1(M_\rib)}) + \rho(M_{J_0},\phi|_{\pi_1(M_{J_0})}) + \rho(\ov{M_\rib},\phi|_{\pi_1(\ov{M_\rib})})  + \rho(\ov{M_L},\phi|_{\pi_1(M_L)})
\end{array}
\end{equation}
We will show the restriction of $\phi$ to $\pi_1(M_{J_0})$ factors nontrivially through $\mathbb{Z}$, so $\rho(M_{J_0},\phi|_{\pi_1(M_{J_0})}) = \rho^0(J_0)$, and the restriction to $\pi_1(\ov{M_L})$ is trivial, so $\rho(\ov{M_L},\phi|_{\pi_1(M_L)})=0$.  This will yield the desired contradiction, similar to the proof of Theorem \ref{thm:distinct}

Since $\pi_1(\ov{M_L})$ is normally generated by its meridian which is isotopic in $Z$ to $\eb$, it suffices to show that $\eb$ is trivial in $\pi_1(Z)^{(1)}/\pi_1(Z)^{(2)}_\p$.  For any space $X$, we denote by $\mathcal{A}^\p(X)$ the localized Alexander module of $X$:
\begin{equation*}
	\mathcal{A}^\p(X) \equiv \alex{X}\otimes \mathbb{Q}[t,t^{-1}]P^{-1} \cong \frac{\pi_1(X)^{(1)}}{\pi_1(X)^{(2)} }\otimes_{\mathbb{Z}[t,t^{-1}]} \mathbb{Q}[t,t^{-1}]P^{-1}.
\end{equation*}
Consider the following diagram where $\phi_*, f_*, g_*, \phi'_*, f_*'$ and $g_*'$ are all induced by inclusion and the vertical maps by projection.  
\begin{equation*}
	\begin{diagram}\dgARROWLENGTH=1.0em
		\node{\mathcal{A}^{\mathbb{Z}}(\csum)} \arrow{s,l}{\psi} \arrow{e,t}{\phi_*}
		\node{\mathcal{A}^{\mathbb{Z}}(V)} \arrow{s,l}{} \arrow{e,t}{f_*}
		\node{\mathcal{A}^{\mathbb{Z}}{(W)}}\arrow{s,l}{} \arrow{e,t}{g_*}
		\node{\mathcal{A}^{\mathbb{Z}}(Z)} \arrow{s,l}{} \arrow{e,t}{}
		\node{\frac{\pi_1(Z)^{(1)}}{\pi_1(Z)^{(2)}_{\mathcal{P}}}} \arrow{sw,l}{i} \\
		\node{\mathcal{A}^\p(\csum)} \arrow{e,t}{\phi'_*}
		\node{\mathcal{A}^\p(V)} \arrow{e,t}{f'_*}
		\node{\mathcal{A}^\p(W)}\arrow{e,t}{g'_*}
		\node{\mathcal{A}^\p(Z)}
	\end{diagram}
\end{equation*}
By definition, $\pi_1(Z)_\p^{(2)}$ is the kernel of $\pi_1(Z)^{(1)}\rightarrow \mathcal{A}^{\p}(Z)$ and so $i$ is injective.  Under the map $\psi$, $\eta_2 \mapsto \eb\otimes 1$.  Since the order of $\eb$, $o_2(t)$, is relatively prime to both $p(t)$ and $p(t^{-1})$,  $o_2(t)\in P$.  Hence 
\begin{equation*}
	\eb\otimes 1 = \eb \cdot o_2(t) \otimes \frac{1}{o_2(t)} = 0,
\end{equation*}
and $\eb$ is trivial in ${\pi_1(Z)^{(1)}}/{\pi_1(Z)^{(2)}_{\mathcal{P}}}$ as desired.

Next consider $\pi_1(M_{J_0})$ which is normally generated by its meridian, $\mu_0$, isotopic in $Z$ to $\ea$.  The kernel of $\psi$ is the $P$-torsion submodule of $\alex{\csum}\cong \alex{K_1}\oplus\alex{K_2}$.  However, $\ea$ is $o_1(t)$-torsion, and $o_1(t) \notin P$ by definition.  Therefore, $\psi(\ea)$ is nontrivial.  Since we assumed $V$ to be slice disk complement for $\csum$, the kernel of $\phi_*'$ is an isotropic submodule of $\mathcal{A}^\p(\csum)$ with respect to the localized Blanchfield form $\bl^\p$ which is given by \cite[Theorem 4.7]{Lei1}
\begin{equation*}
 	\bl^\p(\psi(\ea),\psi(\ea)) = \Psi(\bl(\ea,\ea)).
\end{equation*}
Since this was assumed to be nonzero, $\ea$ must survive in $\mathcal{A}^\p(V)$.  The kernels of both $\pi_1(V)\rightarrow \pi_1(W)$ and $\pi_1(W)\rightarrow \pi_1(Z)$ are normally generated by longitudes of the infecting knots.  These lie in the second term of the derived series of $\pi_1(V)$ and $\pi_1(W)$ and therefore in $\pi_1(V)^{(2)}_\p$ and $\pi_1(W)^{(2)}_\p$.  Hence,
\begin{equation*}
	\mathcal{A}^\p(V)\cong\mathcal{A}^\p(W)\cong\mathcal{A}^\p(Z),
\end{equation*}
and $g_*'\circ f_*'$ is injective.  So $\mu_0$ is a nontrivial element of $\pi_1(Z)^{(1)}/\pi_1(Z)_\p^{(2)}$ and the map 
\begin{equation*}
	\phi: \pi_1(M_{J_0})\rightarrow \frac{\pi_1(Z)}{\pi_1(Z)_\p^{(2)}}
\end{equation*}
must factor through $\pi_1(M_{J_0})/\pi_1(M_{J_0})^{(1)} \cong \mathbb{Z}$.  Therefore, $\rho(M_{J_0},\phi) = \rho_0(J_0)$.  This completes the desired contradiction as Equation \ref{eqn:sigdeffinal} reduces to
\begin{equation*}
	\rho_0(J_0) = -\rho(M_\rib,\phi|_{\pi_1(M_\rib)}) -\rho(\ov{M_\rib},\phi|_{\pi_1(\ov{M_\rib})}) \leq 2 C_\rib.
\end{equation*}

\end{proof}

\bibliographystyle{plain}
\bibliography{diffcurve2}

\begin{thebibliography}{10}

\bibitem{CG1}
A.~J. Casson and C.~McA. Gordon.
\newblock On slice knots in dimension three.
\newblock In {\em Algebraic and geometric topology (Proc. Sympos. Pure Math.,
  Stanford Univ., Stanford, Calif., 1976), Part 2}, Proc. Sympos. Pure Math.,
  XXXII, pages 39--53. Amer. Math. Soc., Providence, R.I., 1978.

\bibitem{CG2}
A.~J. Casson and C.~McA. Gordon.
\newblock Cobordism of classical knots.
\newblock In {\em \`A la recherche de la topologie perdue}, volume~62 of {\em
  Progr. Math.}, pages 181--199. Birkh\"auser Boston, Boston, MA, 1986.
\newblock With an appendix by P. M. Gilmer.

\bibitem{ChGr1}
Jeff Cheeger and Mikhael Gromov.
\newblock Bounds on the von {N}eumann dimension of {$L\sp 2$}-cohomology and
  the {G}auss-{B}onnet theorem for open manifolds.
\newblock {\em J. Differential Geom.}, 21(1):1--34, 1985.

\bibitem{C}
Tim~D. Cochran.
\newblock Noncommutative knot theory.
\newblock {\em Algebr. Geom. Topol.}, 4:347--398, 2004.

\bibitem{CHL5}
Tim~D. Cochran, Shelly Harvey, and Constance Leidy.
\newblock Primary decomposition and the fractal nature of knot concordance.
\newblock {\em Math. Annalen, to appear}.
\newblock preprint Nov 2010: http://front.math.ucdavis.edu/0906.1373.

\bibitem{CHL3}
Tim~D. Cochran, Shelly Harvey, and Constance Leidy.
\newblock Knot concordance and higher-order {B}lanchfield duality.
\newblock {\em Geom. Topol.}, 13:1419--1482, 2009.
\newblock DOI: 10.2140/gt.2009.13.1419.

\bibitem{CHL6}
Tim~D. Cochran, Shelly Harvey, and Constance Leidy.
\newblock 2-torsion in the n-solvable filtration of the knot concordance group.
\newblock {\em Proc. of London Math. Soc.}, 2010.
\newblock first published online August 11, 2010 doi:10.1112/plms/pdq020.

\bibitem{COT}
Tim~D. Cochran, Kent~E. Orr, and Peter Teichner.
\newblock Knot concordance, {W}hitney towers and {$L\sp 2$}-signatures.
\newblock {\em Ann. of Math. (2)}, 157(2):433--519, 2003.

\bibitem{Lei1}
Constance Leidy.
\newblock Higher-order linking forms for knots.
\newblock {\em Comment. Math. Helv.}, 81(4):755--781, 2006.

\bibitem{Le10}
J.~Levine.
\newblock Invariants of knot cobordism.
\newblock {\em Invent. Math.}, 8:98--110, 1969.

\bibitem{R}
Dale Rolfsen.
\newblock {\em Knots and links}, volume~7 of {\em Mathematics Lecture Series}.
\newblock Publish or Perish Inc., Houston, TX, 1990.
\newblock Corrected reprint of the 1976 original.

\bibitem{Ste}
Bo~Stenstr{\"o}m.
\newblock {\em Rings of quotients}.
\newblock Springer-Verlag, New York, 1975.
\newblock Die Grundlehren der Mathematischen Wissenschaften, Band 217, An
  introduction to methods of ring theory.

\bibitem{Str}
Ralph Strebel.
\newblock Homological methods applied to the derived series of groups.
\newblock {\em Comment. Math. Helv.}, 49:302--332, 1974.

\bibitem{Tr1}
H.~F. Trotter.
\newblock On s-equivalence of seifert matrices.
\newblock {\em Invent. Math.}, 20:173--207, 1973.

\bibitem{Tr0}
H.~F. Trotter.
\newblock Knot modules and {S}eifert matrices.
\newblock In {\em Knot theory (Proc. Sem., Plans-sur-Bex, 1977)}, volume 685 of
  {\em Lecture Notes in Math.}, pages 291--299. Springer, Berlin, 1978.

\end{thebibliography}

\end{document}